\documentclass{amsart}
\usepackage{graphicx}
\usepackage{amssymb,amsthm,amstext,amsmath,amscd,latexsym,amsfonts,amscd,enumerate}
\usepackage{epstopdf}
\usepackage[all,cmtip]{xy}
\usepackage[hidelinks=true]{hyperref}
\usepackage[czech,english]{babel}

\usepackage{mathrsfs}

\newcommand{\fa}{\frak{a}}

\newcommand{\fm}{\frak{m}}
\newcommand{\fn}{\frak{n}}
\newcommand{\fp}{\frak{p}}
\newcommand{\fq}{\frak{q}}

\newcommand{\NN}{{\mathbb{N}}}

\newcommand{\ZZ}{\mathbb{Z}}

\newcommand{\PE}{\operatorname{PE}}

\def\RHom{\operatorname{RHom}}

\def\LLambda{\operatorname{{\operatorname{L}}\Lambda}}

\def\depth{\operatorname{depth}}
\def\Lotimes{\otimes^{\operatorname{L}}}

\def\im{\operatorname{im}}

\def\coker{\operatorname{coker}}
\def\ker{\operatorname{ker}}
\def\id{\operatorname{id}}

\def\pd{\operatorname{pd}}

\def\fd{\operatorname{fd}}
\def\id{\operatorname{id}}
\def\pid{\operatorname{p.id}}

\def\invlim{\varprojlim}

\def\fm{\mathfrak m}
\def\Spec{\operatorname{Spec}}

\def\supp{\operatorname{supp}}
\def\cosupp{\operatorname{cosupp}}
\def\Hom{\operatorname{Hom}}

\def\Ext{\operatorname{Ext}}
\def\Tor{\operatorname{Tor}}

\def\cone{\operatorname{cone}}

\def\K{\mathsf{K}}
\def\C{\mathsf{C}}
\def\H{\operatorname{H}}
\def\D{\mathsf{D}}

\def\Flat{\operatorname{\mathsf{Flat}}}

\def\Mod{\mathsf{Mod}}

\newcommand{\tot}{\operatorname{tot}}
\newcommand{\Adelic}{\mathrm{A}}

\numberwithin{equation}{section}

\theoremstyle{plain}
\newtheorem{thm}[equation]{Theorem}          \newtheorem*{thm*}{Theorem}
\newtheorem{prp}[equation]{Proposition}      \newtheorem*{prp*}{Proposition}
\newtheorem{cor}[equation]{Corollary}        \newtheorem*{cor*}{Corollary}
\newtheorem{lem}[equation]{Lemma}            \newtheorem*{lem*}{Lemma}
          \newtheorem*{cnj*}{Conjecture}
\newtheorem{fct}[equation]{Fact}            \newtheorem*{fct*}{Fact}

\theoremstyle{definition}
\newtheorem{dfn}[equation]{Definition}       \newtheorem*{dfn*}{Definition}
\newtheorem{con}[equation]{Construction}     \newtheorem*{con*}{Construction}
     \newtheorem*{funcon*}{Functorial Constructions}
      \newtheorem*{obs*}{Observation}
\newtheorem{rmk}[equation]{Remark}           \newtheorem*{rmk*}{Remark}
\newtheorem{exa}[equation]{Example}          \newtheorem*{exa*}{Example}
         \newtheorem*{exe*}{Exercise}
\newtheorem{qst}[equation]{Question}         \newtheorem*{qst*}{Question}
            \newtheorem{stp*}{Setup}
            \newtheorem*{set*}{Setting}
            \newtheorem*{ntn*}{Notation}

\usepackage{xcolor}

\title[Minimal semi-flat-cotorsion replacements and cosupport]{Minimal semi-flat-cotorsion replacements\\ and cosupport}
\author{Tsutomu Nakamura}
\address[Tsutomu Nakamura]{Graduate School of Mathematics, Nagoya University, Furocho, Chikusaku, Nagoya 464-8602, Japan}
\email{tsutomu.nakamura@math.nagoya-u.ac.jp}

\author{Peder Thompson} 
\address[Peder Thompson]{Institutt for matematiske fag, Norwegian University of Science and Technology, N-7491 Trondheim, Norway}
\email{peder.thompson@ntnu.no}

\date{\today} 



\begin{document}
\maketitle

\begin{abstract}
Over a commutative noetherian ring $R$ of finite Krull dimension, we show that every complex of flat cotorsion $R$-modules decomposes as a direct sum of a minimal complex and a contractible complex. Moreover, we define the notion of a semi-flat-cotorsion complex as a special type of semi-flat complex, and provide functorial ways to construct a quasi-isomorphism from a semi-flat complex to a semi-flat-cotorsion complex. Consequently, every $R$-complex can be replaced by a minimal semi-flat-cotorsion complex in the derived category over $R$. Furthermore, we describe structure of semi-flat-cotorsion replacements, by which we recover classic theorems for finitistic dimensions. In addition, we improve some results on cosupport and give a cautionary example. We also explain that semi-flat-cotorsion replacements always exist and can be used to describe the derived category over any associative ring.
\end{abstract}

\section*{Introduction}
\noindent
The existence of injective envelopes for modules over any ring yields minimal injective resolutions; dually, in settings where projective covers exist---such as for finitely generated modules over a semi-perfect noetherian ring---one can build minimal projective resolutions. These classic forms of minimality are encompassed by the following definition: a complex is minimal if every self homotopy equivalence is an isomorphism; see Avramov and Martsinkovsky \cite{AM02}. In fact, Avramov, Foxby, and Halperin show \cite{AFH} that every complex of injective modules decomposes as a direct sum of a minimal complex and a contractible complex, see also Krause \cite{Kra05}, thus showing every complex has a minimal semi-injective resolution. A dual statement, considered initially by Eilenberg \cite{Eil56}, holds in settings where projective covers exist.

A natural question is whether a complex of flat modules exhibits similar behaviour.  Although flat covers do exist for modules over any ring, due to Bican, El Bashir, and Enochs \cite{BEBE01}, it turns out that minimality is poorly behaved for complexes of flat modules in general: indeed, there exist quasi-isomorphisms between minimal semi-flat complexes that are not isomorphisms of complexes (unlike the case for minimal semi-projective or semi-injective complexes), see for example Christensen and Thompson \cite{CT19}. We thus restrict our focus to complexes of a special type of flat modules: the flat cotorsion modules.

Let $R$ be a commutative noetherian ring. Enochs shows \cite{Eno84} that flat cotorsion $R$-modules---i.e., those flat modules that are also right Ext-orthogonal to flat modules---have a unique decomposition, whose structure is akin to that of injective modules over a noetherian ring as shown by Matlis \cite{Mat58}.  Further, minimality criteria for complexes of flat cotorsion $R$-modules was given by Thompson \cite{Tho19min}. One goal of this paper is to show that when $R$ has finite Krull dimension, such complexes can be decomposed analogously to complexes of injective modules:
\begin{thm*}[See \ref{minsplit} and \ref{mindecomposition}]
Assume $\dim R<\infty$. If $Y$ is a complex of flat cotorsion $R$-modules, then $Y=Y'\oplus Y''$, where $Y'$ is minimal and $Y''$ is contractible.
\end{thm*}

In Section \ref{section_constructions}, we give two functorial approachs to construct a complex of flat cotorsion $R$-modules; one of them builds on work of Nakamura and Yoshino in \cite{NY18}, and the other is inspired by it. We also turn to considering semi-flat-cotorsion complexes, that is, semi-flat complexes of flat cotorsion $R$-modules, as well as replacements by such complexes in the derived category over $R$; see Appendix \ref{appendix_sfc}.  If $F$ is a semi-flat complex, then Constructions \ref{Dim1orCountable} and \ref{DimFiniteExists} yield functorial ways to build a semi-flat-cotorsion complex $Y$ and a quasi-isomorphism $F\to Y$. In particular, we obtain:
\begin{thm*}[See \ref{exists_semi}]
Assume $\dim R<\infty$. Every $R$-complex has a minimal semi-flat-cotorsion replacement in the derived category over $R$.
\end{thm*}
\noindent
Although it is immediate from \cite[Theorem 5.2]{Tho19min} that every $R$-module has a minimal semi-flat-cotorsion replacement without the assumption of finite Krull dimension, the assumption here is natural in considering unbounded complexes.
One motivation for our approach is that not every $R$-module admits a surjection from, or injection to, a flat cotorsion $R$-module---see Example \ref{no_res}---and so our method differs from the one for complexes of injective modules given in \cite[Appendix B]{Kra05}.

In Section \ref{section_AB}, we employ the functorial construction in Construction \ref{DimFiniteExists}, along with the Auslander--Buchsbaum formula, to describe the structure of semi-flat-cotorsion replacements; see Lemma \ref{lambdastructure} and Theorem \ref{structure}. In particular, this extends structure of the minimal pure-injective resolution of a flat module described by Enochs \cite{Eno87}, and also recovers---see Corollary \ref{FFD}---the fact that the finitistic flat dimension of $R$ is at most $\dim R$. In addition, this structure gives a new proof of a classic result of Gruson and Raynaud \cite{GR71} and Jensen \cite{Jen70}: an $R$-module of finite flat dimension has projective dimension at most $\dim R$, see Theorem \ref{FPD}; in particular, the finitistic projective dimension of $R$ is at most $\dim R$ and flat $R$-modules have projective dimension at most $\dim R$.

In Section \ref{section_cosupp}, we apply the other functorial construction, Construction \ref{Dim1orCountable}, in the context of cosupport. The cosupport of an $R$-complex $X$ is the set of prime ideals $\fp$ such that $\RHom_R(\kappa(\fp),X)$ is nontrivial in the derived category over $R$. As an analogue to work of Chen and Iyengar \cite{CI10}, we give in Example \ref{counterexample} an unbounded minimal complex $Y$ of flat cotorsion $R$-modules such that $\cosupp_RY$ is strictly contained in $\bigcup_{i\in \ZZ}\cosupp_RY^i$. This gives a counterexample to \cite[Theorem 2.7]{Tho18cosupp}, unfortunately, and we proceed to give a correction---and improvement---for this result; see Theorem \ref{correct}.

In the appendix, we define the notion of semi-flat-cotorsion replacements for any associative ring $A$, and point to how these complexes can be used to describe the derived category over $A$. In particular, we note that---due to a result of Gillespie \cite{Gil04}---every $A$-complex can be replaced by a semi-flat-cotorsion complex in the derived category over $A$, although minimality remains open; see Question \ref{appendix question0}.
\begin{equation*}
  * \ \ * \ \ *
\end{equation*}
\noindent
Throughout, let $R$ be a commutative noetherian ring. We use standard cohomological notation for $R$-complexes (that is, complexes of $R$-modules), and use $\H(-)$ to denote the cohomology functor. Denote by $\Mod{R}$ the category of $R$-modules, $\C(R)$ the category of $R$-complexes, $\K(R)$ the homotopy category of $R$-complexes, and $\D(R)$ the derived category over $R$. A morphism $\alpha:X\to Y$ in $\C(R)$ or $\K(R)$ is a \emph{quasi-isomorphism} if $\H(\alpha)$ is an isomorphism; an $R$-complex $X$ is \emph{acyclic} if $\H(X)=0$, and is \emph{contractible} if $X$ is isomorphic to the zero complex in $\K(R)$.

\section{Decomposing complexes of flat cotorsion modules}\label{section_decompose}
\noindent
For a complex $P$ of finitely generated free modules over a local ring $(R,\fm,k)$, there exists a decomposition $P=P'\oplus P''$ such that $k\otimes_R P'$ has zero differential and $P''$ is contractible; this was shown in \cite{AFH}. 
Although such a phenomenon does not extend to all complexes of infinitely generated projective modules (see Example \ref{free modules}), there does exist a similar decomposition if we take complexes of $\fm$-adic completions of free modules. In this section, we explain this fact and extend it to the case of complexes of flat cotorsion modules.

We start with the following elementary lemma, in which $R$ is not required to be local. 

\begin{lem}\label{nakayama}
Let $\fa$ be an ideal of $R$, let $T$ and $T'$ be $\fa$-adic completions of projective $R$-modules, and let $\overline{\varphi}:R/\fa\otimes_R T \to R/\fa\otimes_R T'$ be a homomorphism.
\begin{enumerate}
\item[\normalfont{(1)}] There exists a homomorphism $\varphi:T\to T'$ such that $R/\fa\otimes_R \varphi=\overline{\varphi}$.
\item[\normalfont{(2)}] Any such lifting $\varphi:T\to T'$ is an isomorphism if $R/\fa\otimes_R \varphi=\overline{\varphi}$ is an isomorphism.
\end{enumerate}
\end{lem}

\begin{rmk}\label{classic_iso}
Write $T=\invlim_{n\geq 1}(P/\fa^nP)$ for a projective $R$-module $P$. For the proof of the lemma, we recall that there is a natural isomorphism $T/\fa^n T \cong P/\fa^nP$ for each $n\geq 1$. This is well-known for specialists; when $P$ is finitely generated, \cite[\S 8]{Mat89} is sufficient, but even when $P$ is infinitely generated, it is within classic commutative algebra, see \cite[Corollary 2.1.10 and Proposition 2.2.3]{Str90} or \cite[Lemma 2.3]{NY18}. Indeed, it is further known that the isomorphism holds true for any $R$-module, see \cite[Theorem 1.1]{Sim90} or  \cite[Theorem 2.2.5]{Str90}.
\end{rmk}

\begin{proof}[Proof of Lemma \ref{nakayama}]
Set $\phi_1=\overline{\varphi}$.  For $n\geq 1$, we have $T/\fa^{n+1}T$ is a projective $R/\fa^{n+1}$-module by Remark \ref{classic_iso}, hence a map $\phi_n:T/\fa^nT\to T'/\fa^nT'$ lifts to a map $\phi_{n+1}:T/\fa^{n+1}T\to T'/\fa^{n+1}T'$.
Induction yields maps $\phi_n$ for every $n\geq 1$, thus setting $\varphi=\invlim_{n\geq 1} \phi_n$ yields $(1)$.

For $(2)$, let $\varphi:T\to T'$ be any lifting of $\overline{\varphi}$ such that $R/\fa\otimes_R \varphi=\overline{\varphi}$ is an isomorphism. 
Define $\phi_n: T/\fa^nT\to T'/\fa^nT'$ as the map induced by $\varphi$ for $n\geq 1$, where $\phi_1=\overline{\varphi}$.
It is enough to show that each $\phi_n$ is bijective since $\varphi=\invlim_{n\geq 1}\phi_n$. 
We remark that any $R/\fa^n$-module $M$ with $(\fa/\fa^n) M=M$ is zero since the ideal $\fa/\fa^n$ of $R/\fa^n$ is nilpotent. Hence surjectivity of $R/\fa\otimes_R\phi_n=\overline{\varphi}$ implies $\phi_n$ is surjective, and in fact split surjective since $T'/\fa^n T'$ is projective over $R/\fa^n$. Thus injectivity of $R/\fa\otimes_R\phi_n=\overline{\varphi}$ also implies $\ker\phi_n=0$, that is, $\phi_n$ is injective.
\end{proof}

\begin{rmk}\label{completion of flat module}
When $R$ is a local ring with maximal ideal $\fm$ and $T$ is the $\fm$-adic completion of a free $R$-module, the canonical map $T\to T/\fm T$ is a flat cover by \cite[Proposition 4.1.6]{Xu96}, which can instead be used to verify Lemma \ref{nakayama} in this case. 

The argument in the proof of Lemma \ref{nakayama} is inspired by the proof of \cite[II, Proposition 2.4.3.1]{GR71}, which shows that the $\fm$-adic completion of a flat $R$-module is isomorphic to the $\fm$-adic completion of a free $R$-module; see also  \cite[Lemma 6.7.4]{EJ00}.
\end{rmk}

For an index set $A$ and an $R$-module $M$, we denote by $M^{(A)}$ or $\bigoplus_AM$ the direct sum of $A$-copies of $M$. 
If $(R,\fm,k)$ is local, then we write $\widehat{M}$ for the $\fm$-adic completion of $M$.

\begin{lem}\label{trivialsummand}
Assume $(R, \fm, k)$ is a local ring. Let $A$ and $A'$ be some index sets, and let $\partial:\widehat{R^{(A)}}\to \widehat{R^{(A')}}$ be a homomorphism of $R$-modules.  There exist disjoint partitions $A=B\sqcup C$ and $A'=B\sqcup  C'$ and a commutative diagram of $R$-modules
\[\xymatrix{
\widehat{R^{(A)}}\ar[rr]^\partial && \widehat{R^{(A')}}\\
\widehat{R^{(B)}}\oplus \widehat{R^{(C)}} \ar[u]^{\cong}
\ar[rr]_{\scriptsize\begin{bmatrix}1&0\\0&\partial'\end{bmatrix}} && \widehat{R^{(B)}}\oplus \widehat{R^{(C')}}\ar[u]^{\cong}
}
\]
where $k\otimes_R\partial'=0$.
\end{lem}
\begin{proof}
There are isomorphisms $k\otimes_R \widehat{R^{(A)}}\cong k^{(A)}$ and $k\otimes_R \widehat{R^{(A')}}\cong k^{(A')}$, hence we may view $k\otimes_R \partial$ as a linear transformation of $k$-vector spaces. Since $\ker(k\otimes_R \partial)$ and $\im(k\otimes_R\partial)$ are subspaces (and hence direct summands), we may find disjoint partitions $A=B\sqcup C$ and $A'=B\sqcup  C'$ such that the following diagram commutes:
$$\xymatrix{
k^{(A)}\ar[rr]^{k\otimes_R \partial}&& k^{(A')}\\
k^{(B)}\oplus k^{(C)} \ar[u]^{\cong}_{\overline{\alpha}}
\ar[rr]_{\scriptsize\begin{bmatrix}1&0\\0&0\end{bmatrix}} && k^{(B)}\oplus k^{(C')}\ar[u]^{\cong}_{\overline{\beta}}
}$$
The maps $\overline{\alpha}$ and $\overline{\beta}$ lift, by Lemma \ref{nakayama}, to isomorphisms $\alpha:\widehat{R^{(B)}}\oplus \widehat{R^{(C)}}\to \widehat{R^{(A)}}$ and $\beta:\widehat{R^{(B)}}\oplus \widehat{R^{(C')}}\to \widehat{R^{(A')}}$. We thus obtain a commutative diagram: 
$$
\xymatrix{
\widehat{R^{(A)}}\ar[rr]^{\partial}&& \widehat{R^{(A')}}\\
\widehat{R^{(B)}}\oplus \widehat{R^{(C)}} \ar[u]^{\cong}_{\alpha} 
\ar[rr]_{\scriptsize\begin{bmatrix}i&f\\g&h\end{bmatrix}} && \widehat{R^{(B)}}\oplus \widehat{R^{(C')}}\ar[u]^{\cong}_{\beta}
}$$
where $k\otimes_Ri=1$ and $k\otimes_Rf=k\otimes_Rg=k\otimes_Rh=0$. Thus Lemma \ref{nakayama} implies that $i$ is an isomorphism; the conditions on $f$, $g$, and $h$ allow for an elementary translation of the diagram into the desired one.
\end{proof}

We aim to apply Lemma \ref{trivialsummand} to a complex $Y$ of $\fm$-adic completions of free modules. Towards this end, note that application of the lemma to $\partial=d^{0}_Y$ replaces the $4$-term complex $Y^{-1}\to Y^0\to Y^1\to Y^2$ with the following one: 
$$
\xymatrix{
\widehat{R^{(D)}}\ar[rr]^{\scriptsize\begin{bmatrix}0\\a\end{bmatrix}}&& \widehat{R^{(B)}}\oplus \widehat{R^{(C)}} 
\ar[rr]^{\scriptsize\begin{bmatrix}1&0\\0&\partial'\end{bmatrix}} && \widehat{R^{(B)}}\oplus \widehat{R^{(C')}}\ar[rr]^{\scriptsize\begin{bmatrix}0&b\end{bmatrix}}&&\widehat{R^{(D')}},
}$$
where $k\otimes_R \partial'=0$. Hence we can extract 
a direct summand $Y''(0)$ of $Y$ corresponding to a contractible complex $0\to \widehat{R^{(B)}}\xrightarrow{=}\widehat{R^{(B)}}\to 0$. 
By Lemma \ref{trivialsummand}, we can further find such contractible direct summands $Y''(-1)$ and $Y''(1)$ of $Y$ from $d^{-1}_Y$ and $d^{1}_Y$ respectively. Then it is clear from the above matrices that the canonical map $Y''(-1)\oplus Y''(0)\oplus Y''(1)\to Y$ is a split monomorphism. This observation can be used to show the following lemma.

\begin{lem}\label{maxcase 0}
Assume $(R,\fm, k)$ is a local ring. If $Y$ is a complex of  $\fm$-adic completions of free $R$-modules, then $Y=Y'\oplus Y''$, such that the complex $k\otimes_R Y'$ has zero differential and $Y''$ is contractible.
\end{lem}

\begin{proof}
Applying Lemma \ref{trivialsummand} to $d_Y^n:Y^n\to Y^{n+1}$ for each $n\in \mathbb{Z}$, 
extract a contractible direct summand $Y''(n)$ of $Y$ such that the differential of $Y/Y''(n)$ in degree $n$ is zero upon application of $k\otimes_R -$. Then $Y$ has a contractible direct summand of the form $Y''=\bigoplus_{n\in \mathbb{Z}}Y''(n)=\prod_{n\in \mathbb{Z}}Y''(n)$, and  
the differential of $Y'=Y/Y''$ is zero upon application of $k\otimes_R -$.
\end{proof}

The next example exhibits the necessity of taking completions to obtain a suitable decomposition. 
\begin{exa}\label{free modules}
Let $(R,\fm,k)$ be a local ring with $\dim R\geq 1$. Let $x\in \fm$ be an element that is not nilpotent. The localization $R_x$ is therefore nonzero and has a projective resolution of the form $P=(0\to \bigoplus_{\NN}R\to \bigoplus_{\NN}R\to 0)$; indeed, $R_x\cong R[Y]/(1-xY)$ for an indeterminate $Y$, hence the exact sequence
\[\xymatrix{0\ar[r] & R[Y]\ar[r]^{1-xY} & R[Y]\ar[r] &R_x\ar[r] & 0}\]
provides such a resolution $P$. Since $k\otimes_R R_x=0$ and $R_x$ is a flat $R$-module, the complex $k\otimes_R P = (0\to \bigoplus_{\NN}k \xrightarrow{\cong} \bigoplus_{\NN}k\to 0)$ is exact, thus $P$ has no nonzero direct summand $P'$ such that $k\otimes_R P'$ has zero differential. However, $P$ is not contractible since $R_x$ is nonzero.
\end{exa}

The goal of this section is to extend Lemma \ref{maxcase 0} above to the case of complexes of flat cotorsion modules, and so we begin with some basic facts about these. Here we return to the setting of any commutative noetherian ring $R$.

An $R$-module $M$ is {\em flat cotorsion} if it is both flat and cotorsion, that is, $M$ is flat and $\Ext_R^1(F,M)=0$ for every flat $R$-module $F$. Enochs shows in \cite{Eno84} that an $R$-module $M$ is flat cotorsion if and only if $M\cong \prod_{\fp\in \Spec R}T_\fp$, where $T_\fp$ is the $\fp$-adic completion of a free $R_\fp$-module. 
For an ideal $\fa$ of $R$, let $\Lambda^{\fa}=\varprojlim_{n\geq 1} (-\otimes_R R/\fa^n)$ denote the $\fa$-adic completion functor; for an $R$-module $M$, also write $\Lambda^{\fa}M=M^\wedge_\fa$.

A motivation for studying complexes of flat cotorsion $R$-modules is their relationship to cosupport. The notion of cosupport was defined by Benson, Iyengar, and Krause \cite{BIK12}, whose work was inspired by Neeman's \cite{Nee11}. For an $R$-complex $X$, the {\em cosupport} of $X$ is:
$$\cosupp_RX=\{\fp\in \Spec R \mid \H(\RHom_R(\kappa(\fp),X))\not=0\},$$
where $\kappa(\fp)$ stands for the residue field $R_\fp/\fp R_\fp$. See also the equivalent characterizations in \eqref{cosupport characterization}. This is dual to the notion of support defined by Foxby \cite{Fox79}; the \emph{support} of $X$ is:
$$\supp_RX=\{\fp\in \Spec R \mid \H(\kappa(\fp)\Lotimes_R X)\not=0\}.$$
For an index set $A$, we have $\supp \bigoplus_{A} E(R/\fp)\subseteq \{\fp\}$, where $E(R/\fp)$ stands for the injective hull of $R/\fp$ over $R$.
Further, \cite[Theorem 3.4.1(7)]{EJ00} yields an isomorphism
\begin{align}
\textstyle{(\bigoplus_{A}R_\fp)^\wedge_{\fp}\cong \Hom_R(E(R/\fp), \bigoplus_{A} E(R/\fp))}.\label{injective dual}
\end{align}
From this and tensor-hom adjunction, we see that $\cosupp (\bigoplus_{A} R_\fp)^\wedge_\fp\subseteq \{\fp\}$.
Consequently it follows that a flat cotorsion $R$-module $M$ has cosupport contained in a subset $W$ of $\Spec R$ if and only if $M\cong \prod_{\fp\in W}T_\fp$, where $T_\fp$ is the $\fp$-adic completion of a free $R_\fp$-module. We can therefore translate Lemma \ref{maxcase 0} to:

\begin{lem}\label{maxcase}
Let $\fp\in \Spec R$. If $Y$ is a complex of flat cotorsion $R$-modules with $\cosupp_R Y^i\subseteq \{\fp\}$ for every $i\in \ZZ$, then $Y=Y'\oplus Y''$, such that the complex $\kappa(\fp)\otimes_R Y'$ has zero differential and $Y''$ is contractible.
\end{lem}
\begin{proof}
Reduce to a local ring $(R,\fm,k)$; this is just a restatement of Lemma \ref{maxcase 0}.
\end{proof}

For a subset $W$ of $\Spec R$, we define $\dim W$ as the supremum of the lengths of strict chains of prime ideals in $W$. As is standard, $\dim (\Spec R)$ is denoted by $\dim R$; this is the Krull dimension of $R$. The next theorem is the main result of this section. In its proof, we use several basic facts about complexes of flat cotorsion $R$-modules; they are summarized at the end of this section.

\begin{thm}\label{minsplit}
Let $W\subseteq \Spec R$ with $\dim W<\infty$. If $Y$ is a complex of flat cotorsion $R$-modules with $\cosupp_RY^i\subseteq W$ for every $i\in \ZZ$, then $Y=Y'\oplus Y''$, such that the complex $\kappa(\fp)\otimes_R \Hom_R(R_\fp,Y')$ has zero differential for every $\fp\in W$ and $Y''$ is contractible.
\end{thm}
\begin{proof}
We proceed by induction on $\dim W$.
First suppose $\dim W=0$. In this case, $Y\cong \prod_{\fq\in W} \Lambda^\fq Y$ by (\ref{fc_dim0}), and $\Lambda^\fq Y$ consists of flat cotorsion $R$-modules having cosupport contained in $\{\fq\}$ by (\ref{fc_lambda}). For each $\fq\in W$, we apply Lemma \ref{maxcase} to obtain a decomposition $\Lambda^\fq Y=Y'(\fq)\oplus Y''(\fq)$, where $\kappa(\fq)\otimes_R Y'(\fq)$ has zero differential and $Y''(\fq)$ is contractible. Taking a product over $\fq\in W$, we obtain a decomposition 
\begin{align}\label{0splitting}
\textstyle{\prod_{\fq\in W}\Lambda^\fq Y=\prod_{\fq\in W}(Y'(\fq)\oplus Y''(\fq)) \cong  \left(\prod_{\fq\in W}Y'(\fq)\right)\oplus \left(\prod_{\fq\in W}Y''(\fq)\right)}.
\end{align}
A product of contractible complexes is contractible, hence $\prod_{\fq\in W}Y''(\fq)$ is contractible; moreover, \eqref{fc_coloc} implies that for every $\fp\in W$ there is an isomorphism 
\begin{align*}
\textstyle{
\kappa(\fp)\otimes_R\Hom_R(R_\fp, \prod_{\fq\in W}Y'(\fq))\cong \kappa(\fp)\otimes_RY'(\fp),}
\end{align*}
and the latter has zero differential.

Next suppose $\dim W=n>0$. 
Set $Z=\prod_{\fq\in \max W}\Lambda^\fq Y$. By \eqref{degreewise split}, there is a degreewise split exact sequence of complexes of flat cotorsion $R$-modules:
\[\xymatrix{
0\ar[r] & X\ar[r] & Y \ar[r]^{}  & Z\ar[r] & 0.
}\]
The complexes $X$ and $Z$ are complexes of flat cotorsion $R$-modules with cosupport in $W\setminus \max W$ and $\max W$, respectively. As $\dim (W\setminus \max W)<n$ and $\dim (\max W)=0<n$, we may apply the inductive hypothesis to obtain decompositions $X=X'\oplus X''$ and $Z=Z'\oplus Z''$, where $\kappa(\fp)\otimes_R \Hom_R(R_\fp,X')$ and $\kappa(\fp)\otimes_R \Hom_R(R_\fp,Z')$ have zero differential for every $\fp\in W$ and $X''$ and $Z''$ are contractible; see also \eqref{fc_kappa}. Letting $\pi:X\to X'$ be the canonical projection, there exists a complex $P$ of flat cotorsion $R$-modules making the following push-out diagram commute:
\[\xymatrix{
0\ar[r] & X\ar[r]\ar[d]^{\pi}_{} &  Y\ar[r]^{}\ar[d]^f  & Z\ar[r]\ar@{=}[d] & 0\\
0 \ar[r] & X'\ar[r] & P \ar[r] & Z\ar[r] & 0
}\]
The snake lemma yields an exact sequence of complexes of flat cotorsion $R$-modules $0\to X''\to Y\xrightarrow{f} P\to 0$; evidently, this sequence is degreewise split, and it follows from the proof of \cite[Lemma 1.6]{AM02} (see also \cite[Propositions 2.5 and 2.6]{CT19}) that the sequence splits in $\C(R)$ and $f$ is a homotopy equivalence.

On the other hand, letting $\iota:Z'\to Z$ be the canonical inclusion, we obtain a complex $Q$ of flat cotorsion $R$-modules making the pull-back diagram commute:
\begin{align}
\label{induct}
\begin{CD}
\xymatrix{
0\ar[r] & X'\ar[r] &  P\ar[r]^{}  & Z\ar[r] & 0\\
0 \ar[r] & X'\ar[r]\ar@{=}[u] & Q \ar[r]\ar[u]_{g} & Z'\ar[r]\ar[u]^{}_{\iota} & 0
}
\end{CD}
\end{align}
The snake lemma yields a degreewise split exact sequence $0\to Q\xrightarrow{g} P\to Z''\to 0$ of flat cotorsion $R$-modules. As $Z''$ is contractible, this sequence splits in $\C(R)$ and $g$ is a homotopy equivalence by the dual argument of the proof of \cite[Lemma 1.6]{AM02} (see also \cite[Propositions 2.5 and 2.6]{CT19}); let $g':P\to Q$ be a splitting of $g$ in $\C(R)$, and note that $g'$ is also a homotopy equivalence. Thus we have a split exact sequence 
\[\xymatrix{
0\ar[r] & \ker(g'f)\ar[r] & Y\ar[r]^{g'f} & Q\ar[r] & 0,
}\]
where $\ker(g'f)$ is a contractible complex of flat cotorsion modules.

It remains to show that for every $\fp\in W$, the complex $\kappa(\fp)\otimes_R \Hom_R(R_\fp,Q)$ has zero differential.  To do so, we use the degreewise split exact sequence in the bottom row of \eqref{induct}. The modules in $X'$ have cosupport contained in $W\setminus \max W$ as $X'$ is a direct summand of $X$; similarly the modules in $Z'$ have cosupport contained in $\max W$. If $\fp\in \max W$, then $\kappa(\fp)\otimes_R \Hom_R(R_\fp,X')=0$ by \eqref{fc_kappa}. This implies that
$$\kappa(\fp)\otimes_R \Hom_R(R_\fp,Q)=\kappa(\fp)\otimes_R \Hom_R(R_\fp,Z')$$
and that the latter has zero differential by construction.  
If $\fp\in W\setminus \max W$, then we have
$\Hom_R(R_\fp,Z') = 0$ by \eqref{fc_coloc} and hence
$$\kappa(\fp)\otimes_R \Hom_R(R_\fp,X')=\kappa(\fp)\otimes_R \Hom_R(R_\fp,Q),$$
where the former has zero differential by construction.  
\end{proof} 

\begin{rmk}\label{flat cover}
It is known that a complex of objects in an abelian category admitting injective envelopes can be decomposed as a direct sum of a minimal complex and a contractible complex; see \cite[Proposition B.2]{Kra05}. In our situation, however, it is not clear how the arguments of \cite[Proposition B.2]{Kra05} can be employed; indeed, flat envelopes may not exist, and although flat covers do exist over any ring \cite{BEBE01}, there exists a minimal complex of flat cotorsion modules that is not built from flat covers, see Example \ref{not_fc}. This is one motivation for modelling the arguments here on that of finitely generated free modules over a local ring. 

On the other hand, minimality of a complex of flat cotorsion modules with cosupport in $\{\fp\}$ can be characterized by flat covers; see Theorem \ref{criteria}.
\end{rmk}

In the remainder of this section, we summarize several basic facts concerning flat cotorsion $R$-modules which are often used in this paper.
Let $F= \prod_{\fq\in \Spec R}T_\fq$ be a flat cotorsion $R$-module, where $T_\fq$ is the $\fq$-adic completion of a free $R_\fq$-module.  
The derived functors $\LLambda^{\fp}$ and $\RHom_R(R_\fp,-)$ with $\fp\in \Spec R$
are useful for working with such a module; in particular, the following hold:
\begin{align}
&\LLambda^\fp F\cong  \Lambda^\fp F\cong \textstyle{\prod_{\fq\supseteq \fp} T_\fq}\label{fc_lambda};\\
&\RHom_R(R_\fp,F)\cong \Hom_R(R_\fp,F)\cong \textstyle{\prod_{\fq\subseteq \fp} T_\fq.\label{fc_coloc}}
\end{align}
See \cite[\S4, p. 69]{Lip02} and  \cite[Lemma 2.2]{Tho19min}.
Not surprisingly, the above formulas extend to bounded complexes of flat cotorsion $R$-modules, see also \eqref{derived fc_lambda and fc_coloc}. 
If for each $\fq\in \Spec R$ we write $T_\fq=(\bigoplus_{B_\fq} R_\fq)^\wedge_\fq$ for some index set $B_\fq$, then
\begin{align}
\kappa(\fp)\Lotimes_R\RHom_R(R_\fp, F)\cong \kappa(\fp)\otimes_R\Hom_R(R_\fp, F)\cong \textstyle{\bigoplus_{B_\fp}\kappa(\fp)}\label{fc_kappa},
\end{align}
see also Remark \ref{classic_iso}.

For a finitely generated $R$-module $M$, there is also a canonical isomorphism
\begin{align}
M\otimes_R\Hom_R(R_\fp, F)\cong \Hom_R(R_\fp, M\otimes_RF).
\label{fc_tensor}
\end{align}
This map is given by the tensor evaluation map, and we only need to check right exactness of $\Hom_R(R_\fp, -\otimes_R F)$. This verification can be reduced to checking right exactness of $\Hom_R(R_\fp, -\otimes_R T_\fq)$ for each prime $\fq$, as $M$ is finitely presented, which can be checked by using \eqref{injective dual}.

We next explain a useful reduction technique for complexes of flat cotorsion $R$-modules that is used a number of times.
Let $W$ be a subset of $\Spec R$ and $Y$ be a complex of flat cotorsion $R$-modules with $\cosupp Y^i\subseteq W$. We may then write 
\[Y=(\xymatrix{\cdots \ar[r] & \textstyle{\prod_{\fq\in W}T^i_\fq}\ar[r] & \textstyle{\prod_{\fq\in W}T^{i+1}_\fq}\ar[r] & \cdots}),
\]
where $Y^i=\prod_{\fq\in W}T^i_\fq$ and each $T^i_\fq$ is the $\fq$-adic completion of a free $R_\fq$-module.
We denote by $\max W$ the subset of $W$ consisting of  prime ideals which are maximal in $W$ with respect to inclusion. If $\fp \in \max W$, then  
\[
\Lambda^\fp Y=(\xymatrix{\cdots \ar[r] & T^i_\fp\ar[r] & T^{i+1}_\fp\ar[r] & \cdots})
\] 
by \eqref{fc_lambda}. 
Thus, the chain map $Y\to \prod_{\fp\in \max W}\Lambda^\fp Y$ induced by the canonical chain maps $Y\to\Lambda^\fp Y$ yields a degreewise split exact sequence:
\begin{align}\label{degreewise split}
\xymatrix{
0\ar[r] & X\ar[r] & Y \ar[r]^{}  & \textstyle{\prod_{\fp\in \max W}\Lambda^\fp Y}\ar[r] & 0,
}\end{align}
where $X^i=\prod_{\fp\in W\backslash \max W}T^i_\fp$. In particular, if $\dim W=0$, then $\max W=W$ and we have 
\begin{align}\label{fc_dim0}
\textstyle{Y \cong \prod_{\fp\in W} \Lambda^\fp Y}.
\end{align}

\section{Minimality criteria for complexes of flat cotorsion modules}\label{section_min}
\noindent
We now aim to refine and recover \cite[Theorem 3.5]{Tho19min}, which gives minimality criteria for complexes of flat cotorsion modules; our approach uses tools from the previous section. 

Let $Y$ be an $R$-complex. There are important bi-implications about cosupport:
\begin{align}
\fp\in \cosupp_RY&\iff \H(\LLambda^{\fp}\RHom_R(R_\fp, Y))\neq 0\label{cosupport characterization}\\
&\iff \H(\kappa(\fp)\Lotimes_R\RHom_R(R_\fp, Y))\neq 0\nonumber.
\end{align}
These characterizations were essentially shown in \cite{BIK12}; see also \cite[Proposition 4.4]{SWW17}.

The next lemma is a version of \cite[Lemma 3.1]{Tho19min}; the proof given here instead uses the notion of cosupport.

\begin{lem}\label{iso_p}
Let $f$ be a homomorphism of flat cotorsion $R$-modules. The following conditions are equivalent:
\begin{enumerate}
\item[$(1)$] $f$ is an isomorphism.
\item[$(2)$] $\Lambda^\fp\Hom_R(R_\fp,f)$ is an isomorphism for every $\fp\in \Spec R$.
\item[$(3)$] $\kappa(\fp)\otimes_R \Hom_R(R_\fp,f)$ is an isomorphism for every $\fp\in \Spec R$.
\end{enumerate}
\end{lem} 
\begin{proof}
A complex $X$ satisfies $\cosupp_R X=\varnothing$ if and only if $X$ is acyclic\footnote{This is a direct consequence of Neeman's \cite[Theorem 2.8]{Nee92}, which says that $\D(R)$ is generated by the set $\{\kappa(\fp)\mid \fp\in \Spec R\}$.}, see for example \cite[Theorem 4.5]{BIK12}.
By definition, $f$ is an isomorphism if and only if $\cone(f)$ is acyclic. Hence $\cosupp_R(\cone(f))=\varnothing$ if and only if (1) holds. Finally, (\ref{cosupport characterization}) along with \eqref{fc_lambda}, \eqref{fc_coloc}, and \eqref{fc_kappa} yield this is also equivalent to (2) or (3).
\end{proof}

An $R$-complex $X$ is {\em minimal} if every homotopy equivalence $X\to X$ is an isomorphism in $\C(R)$; see \cite{AM02}.  Compare the next result to \cite[Theorems 3.5 and 4.1]{Tho19min}; conditions (2) and (5) here are new. Let us simply denote by $\widehat{R_\fp}$ the $\fp$-adic completion of $R_\fp$.

\begin{thm}\label{criteria}
Let $Y$ be a complex of flat cotorsion $R$-modules. The following conditions are equivalent:
\begin{enumerate}
\item[$(1)$] The complex $Y$ is minimal.
\item[$(2)$] If $Y= Y'\oplus Y''$ and $Y''$ is contractible, then $Y''=0$.
\item[$(3)$] For any $\fp\in \Spec R$, the complex $\kappa(\fp)\otimes_R \Hom_R(R_\fp,Y)$ has zero differential.
\item[$(4)$] For any $\fp\in \Spec R$, the complex $\Lambda^\fp \Hom_R(R_\fp,Y)$ has no direct summand of the form $0\to \widehat{R_\fp}\xrightarrow{=} \widehat{R_\fp}\to 0$.
\item[$(5)$] For any $\fp\in \Spec R$ and $i\in \mathbb{Z}$, the canonical map $T^{i+1} \to \coker(d^{i}_T)$ is a flat cover, where  $T=\Lambda^{\fp} \Hom_R(R_\fp,Y)$.
\end{enumerate}
\end{thm}

\begin{proof}
$(1)\Rightarrow (2)$: This follows by \cite[Proposition 1.7(3)]{AM02}.

$(2)\Rightarrow (3)$: Fix $\fp\in \Spec R$ and set $X=\Hom_R(R_\fp,Y)$. As $X$ is a complex of flat cotorsion $R_\fp$-modules, see \eqref{fc_coloc}, we may apply Theorem \ref{minsplit} to $X$ to obtain a decomposition $X=X'\oplus X''$ such that $\kappa(\fp)\otimes_R X'$ has zero differential and $X''$ is contractible. From the canonical projection $\pi:X\to X'$, form a push-out diagram:
\[\xymatrix{
0\ar[r] & X\ar[r]\ar[d]^{\pi}_{} &  Y\ar[r]^{}\ar[d]  & Y/X\ar[r]\ar@{=}[d] & 0\\
0 \ar[r] & X'\ar[r] & P \ar[r] & Y/X\ar[r] & 0
}\]
As in the proof of Theorem \ref{minsplit}, the snake lemma yields a split exact sequence
\[\xymatrix{
0\ar[r] & X''\ar[r] & Y\ar[r] & P\ar[r] &0.
}\]
The assumption (2) now implies $X''=0$, thus $X=X'$. Hence it holds that
$$\kappa(\fp)\otimes_R\Hom_R(R_\fp,Y) = \kappa(\fp)\otimes_{R} X= \kappa(\fp)\otimes_R X',$$
which has zero differential; (3) follows.

$(3)\Rightarrow (1)$: Let $f:Y\to Y$ be a homotopy equivalence. Thus $\kappa(\fp)\otimes_R \Hom_R(R_\fp,f)$ is also a homotopy equivalence for every $\fp\in \Spec R$. However, since the complex $\kappa(\fp)\otimes_R \Hom_R(R_\fp,Y)$ has zero differential, it follows that $\kappa(\fp)\otimes_R \Hom_R(R_\fp,f)$ is an isomorphism for every $\fp\in \Spec R$. Lemma \ref{iso_p} now yields that $f$ is an isomorphism.

$(3)\Leftrightarrow (4)$: Fix $\fp\in \Spec R$ and set $T=\Lambda^\fp \Hom_R(R_\fp,Y)$. The forward implication follows by replacing $Y$ by $T$ in the implication $(3)\Rightarrow(2)$ already proven above. Conversely, condition $(4)$ forces $d_T^i:T^i\to T^{i+1}$ per Lemma \ref{trivialsummand} to have the property that $\kappa(\fp)\otimes_R d_T^i=0$ for every $i\in \ZZ$.

$(3)\Leftrightarrow (5)$: Fix $\fp\in \Spec R$ and set $T=\Lambda^\fp \Hom_R(R_\fp,Y)$. For each $i\in \ZZ$, apply Lemma \ref{flat cover lemma} below to the exact sequence
\[\xymatrix{
T^{i}\ar[r]^{d_T^{i}} & T^{i+1} \ar[r] & \coker(d_T^i)\ar[r] & 0
}\]
to show that $T^{i+1} \to \coker(d^{i}_T)$ is a flat cover if and only if $\kappa(\fp)\otimes_Rd_T^i=0$.
\end{proof}

\begin{cor}\label{mindecomposition}
Assume $\dim R<\infty$. If $Y$ is a complex of flat cotorsion $R$-modules, then $Y=Y'\oplus Y''$ where $Y'$ is minimal and $Y''$ is contractible.
\end{cor}
\begin{proof}
Apply Theorem \ref{minsplit} and the equivalence $(1)\Leftrightarrow (3)$ of Theorem \ref{criteria}.
\end{proof}

The next lemma is needed for the equivalence $(3)\Leftrightarrow (5)$ in Theorem \ref{criteria} above; notice that its proof shows an $R_\fp$-module $M$ having a presentation by flat cotorsion modules with cosupport in $\{\fp\}$ in fact has a resolution by such modules.
\begin{lem}\label{flat cover lemma}
Let $\fp\in \Spec R$ and let $T^0$ and $T^1$ be $\fp$-adic completions of free $R_\fp$-modules.
Suppose $T^0 \xrightarrow{f} T^1\xrightarrow{g} M\to 0$ is an exact sequence of $R_\fp$-modules. 
The map $g$ is a flat cover of $M$ over $R$ if and only if 
$\kappa(\fp)\otimes_R f=0$.
\end{lem}

\begin{proof}
If $\kappa(\fp)\otimes_R f\not=0$, then Lemma \ref{trivialsummand} implies the complex $T^0\xrightarrow{f} T^1$ has a direct summand $\widehat{R_\fp}\xrightarrow{=}\widehat{R_\fp}$.
The exact sequence $T^0 \xrightarrow{f} T^1\xrightarrow{g} M\to 0$ thus gives a decomposition $g=\begin{bmatrix}0&h\end{bmatrix} :\widehat{R_\fp}\oplus T'\to M$, where $\widehat{R_\fp}\oplus T'=T^1$. The endomorphism $0\oplus \id_{T'}:\widehat{R_\fp}\oplus T'\to \widehat{R_\fp}\oplus T'$ is not an isomorphism, yet it satisfies $g\cdot (0\oplus \id_{T'})=0\oplus h=g$ ; hence $g$ is not a flat cover.

Conversely, suppose that $\kappa(\fp)\otimes_R f=0$; this is equivalent to saying that $\kappa(\fp)\otimes_R g$ is an isomorphism. Suppose that there is a commutative diagram:
$$\xymatrix{
&T^1\ar[ld]\ar[d]^{g}\\
T^1\ar[r]_{g}& M
}
$$
By assumption, all maps in this diagram become isomorphisms upon application of $\kappa(\fp)\otimes_R -$, and so Lemma \ref{nakayama}(2) implies that the map $T^1\to T^1$ is an isomorphism. Hence it remains to show that $g$ is a flat precover, or equivalently, $\ker(g)$ is cotorsion. 
To show this, we will prove that there is an exact sequence 
\begin{align}\label{fc resolution}
\xymatrix{\cdots \ar[r] & T^{-2} \ar[r] & T^{-1} \ar[r] &  T^{0} \ar[r]^{f} & T^{1}\ar[r]^{g} & M\ar[r] & 0},
\end{align}
where all $T^{i}$ are $\fp$-adic completions of free $R_\fp$-modules. Then the truncated complex $\cdots \to T^{-2} \to T^{-1} \to  T^{0}\to 0$ is a resolution of $\ker(g)$, and we can easily verify that $\ker(g)$ is cotorsion, by using Remark \ref{semi-cotorsion property} and \eqref{HomKD2}.

Set $K=\ker(f)$. The $\fp$-adic completion functor induces an isomorphism on both $T^0$ and $T^1$, hence we obtain the following commutative diagram:
\[\xymatrix{
0\ar[r] & K \ar[r]\ar[d] & T^0\ar[r]^f\ar[d]^{=} & T^1\ar[d]^{=}\\
& K^\wedge_\fp \ar[r] & T^0\ar[r]^{f} & T^1
}\] 
By a simple diagram chase, the image of the map $K_\fp^\wedge\to T^0$ is precisely $K$, hence the second row is exact. Choose a surjection from a free $R_\fp$-module $F \to K$; this induces a surjection $F_\fp^\wedge=T^{-1}\to K_\fp^\wedge$ by a standard argument (see the proof of \cite[Theorem 8.1]{Mat89}), hence we obtain a surjection $T^{-1}\to  K$. Repeating this process, we can construct an exact sequence as in \eqref{fc resolution}.
\end{proof}

The existence of the exact sequence \eqref{fc resolution} is also a consequence of a result of Dwyer and Greenlees \cite[Proposition 5.2]{DG02}, which implies that $M\cong \LLambda^{\fp}M$ in this setting.

We end the section with an example showing that statement (5) in Theorem \ref{criteria} may be the best possible in terms of flat covers:
\begin{exa}\label{not_fc}
Let $k$ be an uncountable field and $R=k[x,y]$. The minimal pure-injective resolution of $R$ is a minimal complex of flat cotorsion $R$-modules of the form $0\to P^0\xrightarrow{d^0} P^1\xrightarrow{d^1} P^2\to 0$; see \cite[Remark 3.3 and Theorem 4.8]{Nak19a}. Although $P^2=\coker(d^0)$, the map $d^1:P^1\to P^2$ is not a flat cover. 

A similar example can be constructed for the ring $k[x,y]_{(x,y)}$, using Gruson's \cite[Proposition 3.2]{Gru73}.
\end{exa}

\section{Functorial constructions of semi-flat-cotorsion replacements}\label{section_constructions}
\noindent
In this section, we give two functorial ways to construct a chain map from a complex of flat modules to a complex of flat cotorsion modules such that its mapping cone is pure acyclic; recall that a complex $P$ is \emph{pure acyclic} if $M\otimes_R P$ is acyclic for any $R$-module $M$. In particular, this approach yields a replacement of a semi-flat complex that is both semi-flat and semi-cotorsion (defined below).

Although the setting of this first construction is a bit restricted, the construction itself is not complicated; moreover, it plays a key role in Example \ref{counterexample} below.
\begin{con}\label{Dim1orCountable}
Assume $\dim R\leq 1$ or $R$ is countable.
Let $P$ be an $R$-complex of projective modules. 
Let $W$ be the set of maximal ideals of $R$. The canonical map
$P^i\to \prod_{\fm\in W} \Lambda^{\fm}P^i$ is a pure-injective envelope for each $i\in \ZZ$, see \cite[Remark 6.7.12]{EJ00}.
Moreover, it follows from \cite[Theorem 8.4.12, Corollary 8.5.10]{EJ00}, \cite[II, Corollary 3.3.2]{GR71}, and \cite[Theorem 5.8]{Jen72}
that the pure-injective dimension of $P^i$ is at most $1$ (see also Remark \ref{minimal pure-injective resolutions}) and so there is a short exact sequence of complexes
\[\xymatrix{
0\ar[r] & P \ar[r] & \prod_{\fm\in W} \Lambda^{\fm}P \ar[r] & (\prod_{\fm\in W} \Lambda^{\fm}P)/P \ar[r] & 0
}\]
where every term of $(\prod_{\fm\in W} \Lambda^{\fm}P)/P$ is a flat cotorsion module, see \cite[\S 8.5]{EJ00}. Regarding the above sequence as a double complex, denote its total complex by $X_P$. The rows of this double complex are pure exact, that is, they are exact upon application of $M\otimes_R-$ for any $R$-module $M$. A basic argument \cite[Theorem 12.5.4]{KS06} of double complexes shows that $X_P$ is pure acyclic. On the other hand, there is a commutative diagram
\[\xymatrix{
P \ar[r] \ar[d] & 0\ar[d] \\
\prod_{\fm\in W} \Lambda^{\fm}P \ar[r] & (\prod_{\fm\in W} \Lambda^{\fm}P)/P
}\]
where all arrows express the canonical chain maps.
Regarding both rows as double complexes, this morphism between double complexes naturally induces a chain map $P\to Y_P$, where $Y_P$ denotes the total complex of the second row. The mapping cone of the canonical map $P\to Y_P$ can be identified with $X_P$, hence $P\to Y_P$ is a quasi-isomorphism with pure acyclic mapping cone. Moreover, $Y_P$ is a complex of flat cotorsion $R$-modules.
\end{con}

A complex $P$ is \emph{semi-projective} if $\Hom_R(P,-)$ preserves acyclicity and $P^i$ is projective for every $i\in \ZZ$;
a complex $F$ is \emph{semi-flat} if $-\otimes_RF$ preserves acyclicity and $F^i$ is flat for every $i\in \ZZ$. 
Semi-projective complexes and pure acyclic complexes of flat modules are both semi-flat.
It follows that if $P$ is semi-projective in Construction \ref{Dim1orCountable}, then $Y_P$ is semi-flat. 

A complex $C$ is \emph{semi-cotorsion} if $\Hom_R(-,C)$ preserves acyclicity of pure acyclic complexes of flat modules and $C^i$ is cotorsion for every $i\in \ZZ$ (see Appendix \ref{appendix_sfc}).
By the construction, $Y_P$ consists of flat cotorsion $R$-modules. The next remark shows that $Y_P$ is also semi-cotorsion. 

\begin{rmk}\label{semi-cotorsion property}
Let $\fp\in \Spec R$. As $R$ is noetherian, $\Lambda^\fp$ is left adjoint to the inclusion of $\fp$-adically complete modules into $\Mod{R}$ (this follows from \cite[Theorem 1.1]{Sim90} which implies that $\Lambda^\fp$ is idempotent, see also \cite[Theorem 2.2.5]{Str90} and \cite[\S 4.1]{KS06}); in addition, the functor $-\otimes_R R_\fp$ is left adjoint to the inclusion of $\fp$-local modules into $\Mod{R}$. Hence, if $M$ is any $R$-module and $T_\fp$ is the $\fp$-adic completion of a free $R_\fp$-module, then 
$\Hom_R(M,T_\fp)\cong \Hom_{R}(\Lambda^\fp(M_\fp),T_\fp)$.

Let $T$ be a complex of $\fp$-adic completions of free $R_\fp$-modules. For any pure acyclic complex $X$ of flat $R$-modules, we have $\Hom_R(X,T)\cong \Hom_R(\Lambda^\fp(X_\fp),T)$, and so $\Hom_R(X,T)$ is acyclic because $\Lambda^\fp(X_\fp)$ is contractible.
To see this, we only need to notice that all cycle modules of $X$ are flat, and 
$\Lambda^\fp(-\otimes_R R_\fp)$ sends a short exact sequence of flat $R$-modules to a split short exact sequence of flat cotorsion $R$-modules, see \cite[\S 4, p. 69]{Lip02} and the second paragraph of Remark \ref{completion of flat module}. Therefore $T$ is semi-cotorsion.

Let $W$ be a subset of $\Spec R$ with $\dim W<\infty$ and let $Y$ be a complex of flat cotorsion $R$-modules with $\cosupp Y^i\subseteq W$ for every $i\in \mathbb{Z}$. 
Then we can easily show that $Y$ is semi-cotorsion by an inductive argument on $\dim W$, using the above fact, \eqref{degreewise split}, and \eqref{fc_dim0}. In particular, it follows that all complexes of flat cotorsion $R$-modules are semi-cotorsion when $\dim R<\infty$.

On the other hand, when $R$ is countable but of infinite Krull dimension, we can instead recover  finiteness of projective dimension of flat modules, see Remark \ref{minimal pure-injective resolutions} and \eqref{sup_pd_sup_fd}. From this, a standard argument shows that an acyclic complex of cotorsion modules has cotorsion cycle modules. Consequently any complex of cotorsion modules is semi-cotorsion, as its semi-injective resolution (see Appendix \ref{appendix_sfc}) yields a mapping cone which is semi-cotorsion.
\end{rmk}

We define a complex $Y$ to be \emph{semi-flat-cotorsion} if it is both semi-flat and semi-cotorsion.
The above remark shows that any semi-flat complex of flat cotorsion $R$-modules is semi-flat-cotorsion as long as $R$ is of finite Krull dimension or countable.\footnote{In fact, it follows from \cite{Sto14} or \cite{BIE19} that every complex of flat cotorsion modules is semi-cotorsion without any additional assumptions on the ring, and so every semi-flat complex of flat cotorsion modules is semi-flat-cotorsion; see Lemma \ref{semiflat_cot}. We provide the more elementary observation above for the reader's convenience.}
In particular, if $P$ is assumed to be semi-projective (or semi-flat) in Construction \ref{Dim1orCountable}, then the complex $Y_P$ constructed therein is semi-flat-cotorsion.

Assume $\dim R<\infty$. We now aim to give the construction of a functor from \cite{NY18}, which in particular sends semi-flat $R$-complexes to semi-flat-cotorsion ones. If $W$ is a subset of $\Spec R$ with $\dim W=0$, then write $\bar{\lambda}^W=\prod_{\fp\in W} \Lambda^\fp(-\otimes_R R_\fp)$. There is a canonical morphism $\id_{\C(R)}\to \bar{\lambda}^W$; see \cite[Notation 7.1]{NY18}. For a non-empty subset $W$ of $\Spec R$, a family of subsets $\mathbb{W}=\{W_i\}_{0\leq i \leq n}$ is a \emph{system of slices of $W$} if $W=\bigcup_{0\leq i\leq n}W_i$, the intersections $W_i\cap W_j$ are empty for $i\not=j$, $\dim W_i=0$ for $0\leq i\leq n$, and $W_i$ is specialization-closed in $W$; see \cite[Definition 7.6]{NY18}.

\begin{con}\label{DimFiniteExists}
Assume $\dim R=d<\infty$. Let $W$ be a non-empty subset of $\Spec R$ ordered by inclusion. Denote by $W_0$ the set of maximal elements in $W$. If $W\setminus W_0$ is not empty, then define $W_1$ to be the maximal elements of $W\setminus W_0$. Iterating this process, we obtain a system of slices $\mathbb{W}=\{W_i \mid 0\leq i\leq n\}$ of $W$.  The natural transformations $\id_{\C(R)}\to \bar{\lambda}^{W_i}$  yield (see \cite[Remark 7.3]{NY18}) a \v{C}ech complex of functors:
\[L^{\mathbb{W}}=\left(\xymatrix{
\displaystyle{\prod_{0\leq i\leq n}\bar{\lambda}^{W_{i}}}\ar[r] &
\displaystyle{\prod_{0\leq i<j\leq n}\bar{\lambda}^{W_{j}}\bar{\lambda}^{W_{i}}}\ar[r] &
\cdots \ar[r] &
\displaystyle{\bar{\lambda}^{W_n}\cdots\bar{\lambda}^{W_0}}
}\right).\]
For an $R$-complex $X$, we naturally get a double complex $L^{\mathbb{W}}X$, and the canonical chain maps $X\to \bar{\lambda}^{W_i}X$ induce a morphism $X\to L^{\mathbb{W}}X$ of double complexes. Totalization yields a natural chain map $X\to \tot L^{\mathbb{W}}X$. 

Set $\Adelic^{\mathbb{W}}=\tot L^{\mathbb{W}}$, as in \cite{Nak19}; we see that $\Adelic^{\mathbb{W}}$ is a functor on $\C(R)$ and there is a natural transformation $a^{\mathbb{W}}:\id_{\C(R)}\to \Adelic^{\mathbb{W}}$ (this was written as $\ell^\mathbb{W}$ in \cite{NY18}).

If $M$ is an $R$-module, then $\Adelic^{\mathbb{W}}M=L^{\mathbb{W}}M$.
If $F$ is a flat $R$-module, then the $R$-module $\bar{\lambda}^{W_i}F=\prod_{\fp\in W_i}\Lambda^{\fp}(F_\fp)$ is flat cotorsion, see the second paragraph of Remark \ref{completion of flat module}; thus if $X$ is a complex of flat $R$-modules, 
then $\Adelic^{\mathbb{W}}X$ is a complex of flat cotorsion $R$-modules.

Assume now that $W=\Spec R$, so $d=n$. For each flat $R$-module $X^i$, it follows from \cite[Corollary 7.12]{NY18} that 
$a^{\mathbb{W}}X^i:X^i\to \Adelic^{\mathbb{W}}X^i$ is a (pure) quasi-isomorphism; we give a more elementary proof of this in Fact \ref{quasi_iso_simple} below. Moreover, 
$\cone(a^{\mathbb{W}}X)$ is the totalization of the double complex
\[\xymatrix@C=1.5em{
0\ar[r] & X\ar[r] & \displaystyle{\prod_{0\leq i\leq d}\bar{\lambda}^{W_{i}}}X\ar[r] & \displaystyle{\prod_{0\leq i<j\leq d}\bar{\lambda}^{W_{j}}\bar{\lambda}^{W_{i}}}X\ar[r] &\cdots \ar[r] &  \displaystyle{\bar{\lambda}^{W_d}\cdots\bar{\lambda}^{W_0}}X\ar[r] & 0,
}\]
whose rows are pure exact, and so the totalization $\cone(a^{\mathbb{W}}X)$ is pure acyclic; see for example \cite[Theorem 12.5.4]{KS06}. It then follows that $\Adelic^{\mathbb{W}}$ sends any semi-flat complex to a semi-flat complex of flat cotorsion $R$-modules (cf. \cite[Remark 7.13]{NY18}), that is, a semi-flat-cotorsion complex per Remark \ref{semi-cotorsion property}.
\end{con}

A \emph{semi-flat-cotorsion replacement} of an $R$-complex $X$ is an isomorphism in $\D(R)$ between $X$ and a semi-flat-cotorsion $R$-complex; see Definition \ref{sfcreplacement}.
\begin{thm}\label{exists_semi}
Assume $\dim R<\infty$. Every $R$-complex has a minimal semi-flat-cotorsion replacement in $\D(R)$. 
\end{thm}

\begin{proof}
Let $X$ be an $R$-complex with semi-flat resolution $F\to X$.  Construction \ref{DimFiniteExists} yields a semi-flat-cotorsion replacement $F\xrightarrow{} \Adelic^{\mathbb{W}}F$. By Corollary \ref{mindecomposition}, the complex $\Adelic^{\mathbb{W}}F$ decomposes as $\Adelic^{\mathbb{W}}F=Y'\oplus Y''$ where $Y'$ is a minimal complex of flat cotorsion $R$-modules and $Y''$ is contractible. As $\Adelic^{\mathbb{W}}F$ is semi-flat, so is $Y'$. We then have a diagram of quasi-isomorphisms $Y'\xleftarrow{} F\xrightarrow{} X$, where $Y'$ is a minimal semi-flat-cotorsion $R$-complex.
\end{proof}

\begin{rmk}\label{minimal_env_cov}
Over a commutative noetherian ring, the notion of minimal semi-flat-cotorsion replacements is a common generalization of minimal pure-injective resolutions of flat modules and minimal flat resolutions of cotorsion modules. Indeed, if $M$ is a flat $R$-module, then its minimal pure-injective resolution $P$ (built from pure-injective envelopes) consists of flat cotorsion modules \cite[\S 8.5]{EJ00}, is semi-flat as the mapping cone of $M\to P$ is pure acyclic, and is minimal \cite[Theorem 4.1]{Tho19min}. See also Theorem \ref{criteria} and \cite[Proposition 8.5.26]{EJ00}. Similarly, if $M$ is cotorsion, then its minimal flat resolution $F$ (built from flat covers) consists of flat cotorsion modules \cite[Corollary 5.3.26]{EJ00} and is minimal \cite[Theorem 4.1]{Tho19min}; see also \cite[\S 5.2]{Xu96}. Finally, a minimal semi-flat-cotorsion replacement (if it exists) is unique up to isomorphism in $\C(R)$; see Lemma \ref{sfc_unique}. 
\end{rmk}

In the precedent work \cite{NY18}, the \v{C}ech complex $L^{\mathbb{W}}$ naturally appeared as a consequence of the (generalized) Mayer--Vietoris triangles \cite[Theorem 3.15]{NY18}. 
For the reader's convenience, we provide an alternative proof of the following fact from \cite{NY18}, which we used in Construction \ref{DimFiniteExists}.

\begin{fct}\label{quasi_iso_simple}
Assume $\dim R<\infty$ and let $\mathbb{W}$ be a system of slices of $\Spec R$. If $F$ is a flat $R$-module, then the map $a^{\mathbb{W}}F:F\to \Adelic^{\mathbb{W}}F$ is a quasi-isomorphism. In particular, the mapping cone of $a^{\mathbb{W}}F$ is a pure acyclic complex of flat $R$-modules.
\end{fct}
\begin{proof}
Set $C=\cone(a^{\mathbb{W}}F)$, which by definition is of the following form:
\[\xymatrix@C=1.5em{
0\ar[r] & F\ar[r] & \displaystyle{\prod_{0\leq i\leq n}\bar{\lambda}^{W_{i}}}F\ar[r] & \displaystyle{\prod_{0\leq i<j\leq n}\bar{\lambda}^{W_{j}}\bar{\lambda}^{W_{i}}}F\ar[r] &\cdots \ar[r] &  \displaystyle{\bar{\lambda}^{W_n}\cdots\bar{\lambda}^{W_0}}F\ar[r] & 0.
}\]
The map $a^{\mathbb{W}}F$ is a quasi-isomorphism if and only if $C$ is acyclic, and so it will be enough to show $\supp_RC=\varnothing$; see \cite[Lemma 2.6]{Fox79}. The statement will then follow from the next more general claim, by setting $W=\Spec R$:

\medskip
\noindent
\emph{Claim:} Let $W$ be a non-empty subset of $\Spec R$ with $\dim W=n<\infty$ and let $\mathbb{W}=\{W_i\}_{0\leq i\leq n}$ be a system of slices of $W$. 
If $F$ is a flat $R$-module, then we have $W\cap \supp_RC=\varnothing$, where $C=\cone(a^\mathbb{W}F)$.
\medskip

\noindent
\emph{Proof of Claim:}
We proceed by induction on $n$. If $n=0$, then $C$ is the complex $0\to F\to \bar{\lambda}^W F\to 0$. 
As $R/\fp$ is finitely presented, $R/\fp\otimes_R-$ commutes with the direct product and it follows that $\kappa(\fp)\otimes_R C$ is acyclic for any $\fp\in W$, hence $W\cap \supp_RC=\varnothing$.

Next, suppose that $n>0$. Set $U=\bigcup_{1\leq i\leq n}W_i$ and $U_{i}=W_{i+1}$. We obtain a system of slices $\mathbb{U}=\{U_i\}_{0\leq i\leq n-1}$ of $U$, which yields \v{C}ech complexes $\Adelic^{\mathbb{U}}F$ and $\Adelic^{\mathbb{U}}\bar{\lambda}^{W_0}F$ as in Construction \ref{DimFiniteExists}. 
Set $C'=\cone(a^{\mathbb{U}}F)$ and $C''=\cone(a^{\mathbb{U}}\bar{\lambda}^{W_0}F)$. The canonical map $F\to \bar{\lambda}^{W_0}F$ between flat modules induces a chain map $C'\to C''$ between mapping cones; here the first row is $C'$ and the second row is $C''$:
\[\xymatrix{
0\ar[r]& F\ar[r]\ar[d]& \displaystyle{\prod_{1\leq i\leq n}\bar{\lambda}^{W_{i}}}F\ar[r]\ar[d]&\cdots \ar[r] & \displaystyle{\bar{\lambda}^{W_{n}}\cdots\bar{\lambda}^{W_1}}F\ar[r]\ar[d]&0\\
0\ar[r]& \bar{\lambda}^{W_0}F\ar[r]& \displaystyle{\prod_{1\leq i\leq n}\bar{\lambda}^{W_{i}}}\bar{\lambda}^{W_{0}}F\ar[r]&\cdots \ar[r] & \displaystyle{\bar{\lambda}^{W_{n}}\cdots\bar{\lambda}^{W_0}}F\ar[r]&0
}\]
If we regard the above diagram as a double complex, then its total complex is $C$.
Thus to show that $W\cap \supp_R C=\varnothing$, it is enough to justify:
\begin{enumerate}
\item[(i)] If $\fp\in U$, then $\kappa(\fp)\otimes_R C'$ and $\kappa(\fp)\otimes_RC''$ are acyclic.
\item[(ii)] If $\fp\in W_0$, then $\kappa(\fp)\otimes_R-$ transforms all vertical maps into isomorphisms.
\end{enumerate}
As $\dim U=n-1$, the inductive hypothesis implies (i). For $\fp\in W_0$, application of $\kappa(\fp)\otimes_R-$ to the above diagram leaves only the left column nonzero, which becomes an isomorphism by the argument for the $n=0$ case above, thus (ii) also holds.
\end{proof}

The construction of $\cone(a^{\mathbb{W}}F)$ as a totalization of a double complex above is just an analogue of the corresponding construction of classic (extended) \v{C}ech complexes: For a sequence $x_1,...,x_n\in R$, the \v{C}ech complex $\check{C}(x_1,...,x_n)$ (see \cite[\S 5.1]{BS98}) is naturally isomorphic to $\check{C}(x_1,...,x_{n-1})\otimes_R \check{C}(x_n)$. Note however, that one must be a bit cautious: If $X$ is an $R$-complex of finitely generated modules, then $(\Adelic^{\mathbb{W}}R)\otimes_RX\cong \Adelic^{\mathbb{W}}X$, see \cite[(8.4)]{NY18}, but this isomorphism need not hold for an arbitrary $R$-complex $X$. Moreover, $\bar{\lambda}^{W_j}\bar{\lambda}^{W_i}$ need not be isomorphic to $\bar{\lambda}^{W_i}\bar{\lambda}^{W_j}$.

\begin{rmk}\label{minimal pure-injective resolutions}
If $\dim R<\infty$, then the minimal pure-injective resolution of a flat module, constructed as in Construction \ref{DimFiniteExists} and using Corollary \ref{mindecomposition}, implies immediately that the pure-injective dimension of any flat $R$-module is at most $\dim R$, see also Remark \ref{minimal_env_cov} and \cite[\S 2, Corollary]{Eno87}. Recall that Construction \ref{Dim1orCountable} uses this fact under the assumption $\dim R\leq 1$; this case is enough for one of the main aims in Section \ref{section_cosupp}, see Example \ref{counterexample}.

On the other hand, Construction \ref{Dim1orCountable} also treats any countable ring $R$ (not only those of finite dimension), as the pure-injective dimension of any flat $R$-module is at most $1$ for such rings. This follows as a consequence of Lemma \ref{pd_fd_sups_lemma} below as for each $\fp\in \Spec R$ the $R$-module $R_\fp$ admits a projective resolution similar to that of Example \ref{free modules}; see also \cite[Lemma 2.12]{GT12}.
\end{rmk}

For any commutative noetherian ring $R$, we denote projective dimension by $\pd_R$ and pure-injective dimension by $\pid_R$.

\begin{lem}\label{pd_fd_sups_lemma}
One has the following equality:
$$\sup\{\pd_R R_\fp \mid \fp\in \Spec R\} = \sup\{\pid_R F \mid F \in \Flat{R}\}.$$
\end{lem}

Before proving this lemma, recall that 
\begin{align}\label{sup_pid_sup_pd}
\sup\{\pid_R F \mid F \in \Flat R\}=\sup\{\pd_R F\mid F\in \Flat R \},
\end{align} 
see \cite[Theorem 8.4.12]{EJ00}.
Hence Lemma \ref{pd_fd_sups_lemma} along with \eqref{sup_pid_sup_pd} yields
\begin{align}\label{sup_pd_sup_fd}
\sup\{\pd_R R_\fp \mid \fp\in \Spec R\} = \sup\{\pd_R F \mid F\in \Flat R\}.
\end{align}
A similar equality to \eqref{sup_pd_sup_fd} was originally shown by Gruson and Raynaud \cite[II, Theorem 3.3.1]{GR71} treating all multiplicatively closed subsets. It was strengthened by Enochs \cite[\S 3, Corollary 1]{Eno87} showing \eqref{sup_pd_sup_fd}, provided that the right hand side is finite, see also \cite[Proposition 8.5.13]{EJ00}. Our Lemma \ref{pd_fd_sups_lemma} does not need any such finiteness, and furthermore it recovers these existing results via \eqref{sup_pid_sup_pd}.

\begin{proof}[Proof of Lemma \ref{pd_fd_sups_lemma}]
Set $n=\sup\{\pd_RR_\fp \mid \fp\in \Spec R\}$. First, the inequality $n\leq \sup\{\pid_R F \mid F \in \Flat{R}\}$ is clear by virtue of \eqref{sup_pid_sup_pd}.
Hence the equality follows trivially when $n$ is infinite. Assume $n<\infty$ and fix a flat $R$-module $F$. Our goal is to show $\pid_R F\leq n$.
Let $C=(0\to \PE^0(F)\to \PE^1(F)\to \cdots)$ be a minimal pure-injective resolution of $F$, and note that $C$ is a minimal complex of flat cotorsion modules, see Remark \ref{minimal_env_cov}.
Take any $\fp \in \Spec R$. 
In view of \eqref{fc_kappa} and Theorem \ref{criteria}, we only need to verify that $\H^i(\kappa(\fp)\otimes_R\Hom_R(R_\fp, C))=0$ for $i>n$.
For this purpose, we use a natural isomorphism in the category of complexes
$$\kappa(\fp)\otimes_R\Hom_R(R_\fp, C)\cong\Hom_{R/\fp}(\kappa(\fp), R/\fp\otimes_R C);$$
one can deduce this from \eqref{fc_tensor}. Now, by \cite[Theorem 8.5.1]{EJ00}, the complex $R/\fp\otimes_R C$ is a minimal pure-injective resolution of $R/\fp\otimes_RF$ over $R/\fp$, so it holds that 
\begin{align*}
\textstyle{
\RHom_{R/\fp}(\kappa(\fp), R/\fp\otimes_R F)\cong\Hom_{R/\fp}(\kappa(\fp), R/\fp\otimes_R C)}
\end{align*}
in $\D(R/\fp)$, since $\kappa(\fp)=R_\fp/\fp R_\fp$ coincides with the quotient field of $R/\fp$. Noting that  $\pd_{R/\fp} {\kappa(\fp)}\leq n$ as $\pd_R R_\fp \leq n$, we have 
$$\H^i(\RHom_{R/\fp}(\kappa(\fp), R/\fp\otimes_R F))\cong \H^i(\Hom_{R/\fp}(\kappa(\fp), R/\fp\otimes_R C))=0$$
for $i>n$, as desired.
\end{proof}

The next example shows the necessity for considering not semi-flat-cotorsion resolutions but semi-flat-cotorsion replacements.

\begin{exa}\label{no_res}
Let $k$ be a field, let $R=k[x,y]_{(x,y)}$, and let $M=R/(x^2)$. We show there does not exist a complex of flat cotorsion $R$-modules having a quasi-isomorphism $M\to Y$ or a quasi-isomorphism $Y\to M$.

As $x^2M=0$, and every flat $R$-module is torsion-free, there are no nonzero homomorphisms from $M$ to a flat $R$-module. This forbids the existence of a complex of flat cotorsion $R$-modules having a quasi-isomorphism $M\to Y$.

We next consider homomorphisms from a flat cotorsion module $F=\prod_{\fp\in \Spec R}T_\fp$ to $M$, where each $T_\fp$ is the $\fp$-adic completion of a free $R_\fp$-module for $\fp\in \Spec R$. As $\Lambda^{(x)}M\cong M$, we immediately have by Remark \ref{semi-cotorsion property} and \eqref{fc_lambda} an isomorphism $\Hom_R(F,M)\cong \Hom_R(T_{(x)}\oplus T_{(x,y)},M)$. Fix $f\in \Hom_R(T_{(x)},M)$ and let $a\in T_{(x)}$. As the image of $y^n$ is invertible in $T_{(x)}$ for all $n\geq 1$, we obtain that $f(a)=y^nf(a/y^n)\in (x,y)^nM$ for all $n\geq 1$. Krull's intersection theorem then yields that $f(a)\in \bigcap_{n\geq 1}(x,y)^nM=0$. As $a\in T_{(x)}$ is arbitrary, this shows $f=0$, hence $\Hom_R(T_{(x)},M)=0$. 

Therefore, if there exists a complex $Y$ of flat cotorsion $R$-modules with a quasi-isomorphism $Y\to M$, there must be a surjection $T_{(x,y)}\to M$. However, ideal-adic completion preserves surjectivity of morphisms by a standard argument (see the proof of \cite[Theorem 8.1]{Mat89}), and so this would imply that the map $T_{(x,y)}\to \Lambda^{(x,y)}M$ induced by $\Lambda^{(x,y)}$ is surjective and factors through $M$, contradicting the fact that $M\to \Lambda^{(x,y)}M$ is not surjective.
\end{exa}

Indeed, semi-flat-cotorsion replacements always exist over any ring. This is almost directly deduced from Gillespie's work \cite{Gil04}, which shows that pure acyclic complexes of flat modules and semi-cotorsion complexes form a complete cotorsion pair; see Theorem \ref{Gillespie_exist}. 
However, we do not know whether minimal ones can be  always obtained as in Theorem \ref{exists_semi}, see Question \ref{appendix question0}.   
In addition, the constructions here yield additional information about the structure of semi-flat-cotorsion replacements, which we take advantage of in the next section.

\section{Structure of semi-flat-cotorsion replacements and finitistic dimensions}\label{section_AB} 
\noindent
The goal of this section is to describe the structure of semi-flat-cotorsion replacements using the construction of the functor $\Adelic^{\mathbb{W}}$ in Section \ref{section_constructions}, and give applications of this structure to finitistic flat and projective dimensions. 

If $\dim R=d<\infty$, we set $W_i=\{\fp\in \Spec R \mid \dim R/\fp = i\}$, and notice that $\mathbb{W}=\{W_i\}_{0\leq i \leq d}$ is a system of slices for $\Spec R$; see Section \ref{section_constructions}. In this setting, the functor $\Adelic^{\mathbb{W}}$ is now defined as in Construction \ref{DimFiniteExists}; it sends semi-flat complexes to semi-flat-cotorsion complexes. 

\begin{lem}\label{lambdastructure}
Assume $\dim R=d<\infty$. If $F$ is a complex of flat $R$-modules with $F^i=0$ for $i>0$, then $\Adelic^{\mathbb{W}}F$ has the form:
\[\xymatrix@C=.8em{
\cdots \ar[r] 
& \displaystyle{\prod_{\fp\in \Spec R}T_\fp^{-1}} \ar[r] 
& \displaystyle{\prod_{\fp\in \Spec R}T_\fp^{0}}\ar[r] 
& \displaystyle{\prod_{\dim R/\fp\geq 1}T_\fp^{1}}\ar[r] 
& \cdots \ar[r] 
&  \displaystyle{\prod_{\dim R/\fp\geq d}T_\fp^{d}}\ar[r] 
& 0\ar[r] 
& \cdots
}\]
where each $T_\fp^n$ is the $\fp$-adic completion of a free $R_\fp$-module. Moreover, if $n\in \ZZ$ and $F^i=0$ for $i<n$, then $(\Adelic^{\mathbb{W}}F)^i=0$ for $i<n$. 
\end{lem}
\begin{proof}
This is a direct consequence of the construction of $\Adelic^{\mathbb{W}}$.
\end{proof}

Set $\inf X=\inf\{i\mid \H^i(X)\not=0\}$ and $\sup X=\sup\{i\mid \H^i(X)\not=0\}$ for an $R$-complex $X$; if $\H(X)=0$ then set $\inf X=\infty=\inf \varnothing$ and $\sup X=-\infty=\sup \varnothing$. 

If $(R,\fm,k)$ is local and $X$ is an $R$-complex that is isomorphic in $\D(R)$ to a bounded complex of flat $R$-modules, then the following is a version of the Auslander--Buchsbaum formula:
\begin{align}\label{AB}
\depth_RX=\depth_R R+\inf(k\Lotimes_R X).
\end{align} 
This is a special case of the generalization given by Foxby and Iyengar \cite[Theorem 2.4]{FI03}; in the case $\H(X)=0$, the equality \eqref{AB} trivially holds. Here one may define $\depth_R X=\inf \RHom_R(k,X)$; see \cite[Theorem 2.1 and Definition 2.3]{FI03}.

If $\fp\in \Spec R$ and $X$ is an $R$-complex that is isomorphic in $\D(R)$ to a bounded complex of flat cotorsion $R$-modules, then an immediate consequence of \eqref{fc_coloc} and \eqref{AB} is an equality:
\begin{align}\label{AB_prime}
\depth_{R_\fp} \RHom_R(R_\fp,X)=\depth_{R_\fp}R_\fp+\inf(\kappa(\fp)\Lotimes_{R_\fp} \RHom_R(R_\fp,X)).
\end{align} 

Recall that $\widehat{R_\fp}$ stands for the $\fp$-adic completion of $R_\fp$. The first author noticed the formulation of the next lemma through a collaboration with Takahashi and Yassemi \cite{NTY}.
\begin{lem}\label{Rpsup}
Assume $\dim R<\infty$ and let $\fp\in \Spec R$. Let $X$ be an $R$-complex that is isomorphic in $\D(R)$ to a bounded complex of flat $R$-modules. If $Y$ is a minimal semi-flat-cotorsion replacement of $X$, then 
\begin{align*}\depth_{R_\fp}R_\fp-\depth_{R_\fp}\RHom_R(R_\fp,X)&=\sup \{i\mid \widehat{R_\fp} \text{ is a direct summand of }Y^{-i}\}.
\end{align*}
\end{lem}
\begin{proof}
Let $Y$ be a minimal semi-flat-cotorsion replacement of $X$; notice that $Y$ is a bounded complex of flat cotorsion $R$-modules, by Corollary \ref{mindecomposition}, Lemma \ref{lambdastructure} and Lemma \ref{sfc_unique}. As $\kappa(\fp)\Lotimes_{R_\fp} \RHom_R(R_\fp,X) \cong  \kappa(\fp)\otimes_{R_\fp} \Hom_R(R_\fp,Y)$, and the latter complex has zero differential for every $\fp\in \Spec R$ by Theorem \ref{criteria}, we have
\begin{align*}
\inf(\kappa(\fp)\Lotimes_{R_\fp} \RHom_R(R_\fp,X))
&=\inf\{i\mid \widehat{R_\fp} \text{ is a direct summand of }Y^i\},
\end{align*}
where the right hand side is 
$-\sup\{i\mid \widehat{R_\fp} \text{ is a direct summand of }Y^{-i}\}$.
The claim now follows from (\ref{AB_prime}). 
\end{proof}

\begin{rmk}\label{Enochs_Xu_number}
If an $R$-complex $X$ has a minimal semi-flat-cotorsion replacement $Y$ with $Y^i=0$ for $i\gg 0$, then one has $Y^i=\prod_{\fp\in \Spec R} T_\fp^i$ with $T_\fp^i=(\bigoplus_{B_\fp^i}R_\fp)^{\wedge}_\fp$ and $B_\fp^i=\dim_{\kappa(\fp)}\H^i(\kappa(\fp)\Lotimes_R\RHom_R(R_\fp, X))$, by \eqref{fc_kappa}, Theorem \ref{criteria}, and \eqref{derived fc_lambda and fc_coloc}.
When $X$ is a cotorsion $R$-module, $Y$ is nothing but a minimal flat resolution of $X$ (see Remark \ref{minimal_env_cov}), and $B_\fp^{-i}=\dim_{\kappa(\fp)}\Tor^{R_\fp}_{i}(\kappa(\fp),\Hom_R(R_\fp, X))$ is the $i$th \emph{dual Bass number} in the sense of Enochs and Xu \cite{EX97}.
\end{rmk}

Let $X$ be an $R$-complex.  
The {\em flat dimension} of $X$ is defined as
\begin{align*}
\fd_RX
&=\inf\left\{\sup\{i\mid F^{-i}\not=0\} \,\middle|\, X\cong F\text{ in $\D(R)$, with $F$ semi-flat}\right\}.
\end{align*}
In the next theorem, we simply write $\depth R_\fp$ for $\depth_{R_\fp}R_\fp$.
\begin{thm}\label{structure}
Assume $\dim R=d<\infty$. If $M$ is an $R$-module with $\fd_RM<\infty$, then the minimal semi-flat-cotorsion replacement of $M$ has the following form:
\[\resizebox{\displaywidth}{!}{
\xymatrix@C=1em{
0\ar[r] 
& \displaystyle{\prod_{\depth R_\fp\geq d}T_\fp^{-d}} \ar[r] 
& \cdots \ar[r]
& \displaystyle{\prod_{\depth R_\fp\geq 1}T_\fp^{-1}} \ar[r] 
& \displaystyle{\prod_{\fp\in \Spec R}T_\fp^{0}}\ar[r] 
& \displaystyle{\prod_{\dim R/\fp\geq 1}T_\fp^{1}}\ar[r] 
& \cdots \ar[r] 
&  \displaystyle{\prod_{\dim R/\fp\geq d}T_\fp^{d}}\ar[r] 
& 0.
}}\]
\end{thm}
\noindent

\begin{proof}
Let $M$ be an $R$-module with $\fd_RM<\infty$. The module $M$ has a minimal semi-flat-cotorsion replacement $Y$; moreover, it is isomorphic to a direct summand of the one in Lemma \ref{lambdastructure} by Corollary \ref{mindecomposition} and Lemma \ref{sfc_unique}, hence $Y$ is bounded.
Write $Y^i=\prod_{\fp\in \Spec R} T_\fp^i$, where $T_\fp^i$ is the $\fp$-adic completion of a free $R_\fp$-module.
Fix $\fp\in \Spec R$. It now follows from Lemma \ref{Rpsup}, because $\depth_{R_\fp} \RHom_R(R_\fp,M)\geq 0$, that $T_{\fp}^{-i}=0$ for $i>\depth_{R_\fp} R_\fp$ as desired, where we just have $T_{\fp}^{-i}=0$ for all $i\in \mathbb{Z}$ if $\depth_{R_\fp} \RHom_R(R_\fp,M)=\infty$. 
\end{proof}

Theorem \ref{structure} specializes to give the structure of a minimal pure-injective resolution of a flat module shown in \cite[Theorem 2.1]{Eno87} provided $\dim R<\infty$. Also (perhaps unsurprisingly) this implies that the finitistic flat dimension of $R$, defined to be $\sup\{\fd_RM\mid M \text{ is an $R$-module with } \fd_RM<\infty\}$, is at most $\dim R$; this was shown by Auslander and Buchsbaum \cite[Theorem 2.4]{AB58}:
\begin{cor}[Auslander and Buchsbaum]\label{FFD}
The finitistic flat dimension of $R$ is at most $\dim R$.
\end{cor}
\begin{proof} 
Immediate by Theorem \ref{structure}.
\end{proof}

Compare the next result with \cite[Corollary 5.9]{CIM19}.

\begin{thm}\label{compute_fd}
Assume $\dim R<\infty$. If $X$ is an $R$-complex that is isomorphic in $\D(R)$ to a bounded complex of flat $R$-modules, then 
$$\fd_RX=\sup\{\depth_{R_\fp} R_\fp-\depth_{R_\fp} \RHom_R(R_\fp,X) \mid \fp\in \Spec R\}.$$
\end{thm}
\begin{proof}
If $\H(X)=0$, then $\fd_RX=-\infty$ and $\depth_{R_\fp} \RHom_R(R_\fp,X)=\infty$ for each $\fp\in \Spec R$. Hence the above equality holds. Suppose that $\H(X)\neq 0$. Set $n=\fd_RX$ and let $F$ be a semi-flat complex isomorphic to $X$ in $\D(R)$ and satisfying both $F^{i}=0$ for $i<-n$ and $F^i=0$ for $i\gg0$. We obtain from Lemma \ref{lambdastructure} that the semi-flat-cotorsion replacement $\Adelic^{\mathbb{W}}F$ of $F$ satisfies $(\Adelic^{\mathbb{W}}F)^i=0$ for $i<-n$. In other words, we have $n\geq \sup\{i\mid (\Adelic^{\mathbb{W}}F)^{-i}\not=0\}$. Further, we can by Corollary \ref{mindecomposition} find a minimal semi-flat-cotorsion replacement $Y$ of $X$ as a direct summand of $\Adelic^{\mathbb{W}}F$.
Then we have $n\geq \sup\{i\mid Y^{-i}\not=0\}$, and this must be an equality since $\fd_RX=n$.
It then holds that 
\begin{align*}
n&=\sup\{i\mid Y^{-i}\not=0\}\\
&=\sup\left\{\sup\{i \mid \widehat{R_\fp} \text{ is a direct summand of $Y^{-i}$}\}\,\middle|\,\fp\in \Spec R\right\}\\
&=\sup\{\depth_{R_\fp} R_\fp - \depth_{R_\fp}\RHom_R(R_\fp,X)\mid \fp\in \Spec R\},
\end{align*}
where the last equality follows from Lemma \ref{Rpsup}.
\end{proof}

We end the section by recovering two classic facts: First, that the finitistic projective dimension of $R$ is at most $\dim R$---this was originally proven by Gruson and Raynaud \cite[II, Theorem 3.2.6]{GR71}---and second, that flat $R$-modules have projective dimension at most $\dim R$---this is due to Gruson and Raynaud \cite[II, Theorem 3.2.6]{GR71} and Jensen \cite[Proposition 6]{Jen70}.

\begin{thm}[Gruson--Raynaud, Jensen]\label{FPD}
If an $R$-module has finite flat dimension, then its projective dimension is at most $\dim R$.
\end{thm}
\begin{proof}
We may assume $d=\dim R$ is finite. 
Let $M$ be an $R$-module with $\fd_RM<\infty$ and let $N$ be any $R$-module. It is sufficient to show $\Ext_R^i(M, N)=0$ for all $i>d$. Take a flat resolution $F$ of $N$ and replace it by  $Y=\Adelic^{\mathbb{W}}F$ in $\D(R)$; this is a right bounded complex of flat cotorsion modules as described in Lemma \ref{lambdastructure}.
Our goal is to show that $\H^i(\RHom_R(M,Y))=0$ for $i>d$.

As above, set $W_i=\{\fp\in \Spec R \mid \dim R/\fp = i\}$.
By iteration of \eqref{degreewise split}, we can make a sequence of subcomplexes 
$$0=Y_{d+1}\subset Y_{d}\subset Y_{d-1}\subset\cdots \subset Y_{1}\subset Y_{0}=Y,
$$
where each quotient complex $Y_{i}/Y_{i+1}$ is a complex of flat cotorsion modules with cosupport in $W_{i}$; that is, $Y_{i}/Y_{i+1}$ is isomorphic to $\prod_{\fp\in W_{i}} T(\fp)$, where for each prime $\fp\in \Spec R$, we have $T(\fp)=\Lambda^\fp\Hom_R(R_\fp,Y)$ is a complex of flat cotorsion modules with cosupport in $\{\fp\}$, see \eqref{fc_lambda} and \eqref{fc_coloc}.
Thus it is enough to show that
$\H^i(\RHom_R(M,T(\fp)))=0$ for $i>d$ and $\fp\in \Spec R$. Note that $T(\fp)^i=0$ for $i>\dim R/\fp$ by  Lemma \ref{lambdastructure}.

Now, since $T(\fp)$ is a complex of $R_\fp$-modules, we have 
$$\RHom_R(M,T(\fp))\cong \RHom_{R}(M_\fp,T(\fp)).$$
By Corollary \ref{FFD}, the module $M_\fp$ has a flat resolution $P$ over $R_\fp$ such that $P^i=0$ for $i<-\dim R_\fp$.  
Since $P$ is semi-flat (over $R$) and $T(\fp)$ is semi-cotorsion by Remark \ref{semi-cotorsion property}, we have  
$$\RHom_{R}(P, T(\fp))\cong \Hom_{R}(P, T(\fp)),$$
see \eqref{HomKD2}. Combining these two isomorphisms, we obtain 
$$\H^i(\RHom_R(M,T(\fp)))\cong \H^i(\Hom_{R}(P, T(\fp)))\cong \Hom_{\K(R)}(P, T(\fp)[i])=0$$
for $i>d\geq \dim R/\fp+\dim R_\fp$, as desired.
\end{proof}

Indeed, Remark \ref{minimal pure-injective resolutions} and \eqref{sup_pid_sup_pd} are available as long as one verifies the inequality $\pd_R F\leq \dim R$ for a flat $R$-module $F$; this approach is close to that of Enochs \cite[Proposition 3.1]{Eno87}. The inequality was also recovered in \cite[\S 4]{NY18} using a different approach, which inspired our new proof to simultaneously recover these two classic facts.

\section{Cosupport: a refinement, correction, and counterexample}\label{section_cosupp}
\noindent
The goal of this section is to examine the relationship between the cosupport of an $R$-complex and the prime ideals appearing in a minimal semi-flat-cotorsion replacement; in particular, we seek to correct and improve \cite[Theorem 2.7]{Tho18cosupp}. 

As minimal semi-flat-cotorsion replacements exist at least for rings of finite Krull dimension, we will compare the cosupport of a minimal complex $Y$ of flat cotorsion $R$-modules to the set $\bigcup_{i\in \ZZ} \cosupp_R Y^i$, which can be thought of as the prime ideals appearing in $Y$. 
We begin with a lemma showing one containment always holds:
\begin{lem}\label{cosuppinclusion}
Let $Y$ be a complex of flat cotorsion $R$-modules. The inclusion holds:
$$\cosupp_RY\subseteq \bigcup_{i\in\mathbb{Z}}\cosupp_RY^i.$$
\end{lem}

This is proved in \cite[Proposition 6.3]{NY18} under several conditions; for example, if $\dim R<\infty$, or if $Y^i=0$ for $i\ll0$, then \cite[Proposition 6.3]{NY18} implies the above inclusion. 
In general, for $\fp\in \Spec R$ and a complex $Y$ of flat cotorsion $R$-modules, one has
\begin{align}\label{derived fc_lambda and fc_coloc}
&\RHom_R(R_\fp,Y)\cong \Hom_R(R_\fp,Y)\quad\text{and}\quad  \LLambda^{\fp}Y\cong \Lambda^{\fp}Y.
\end{align}
The first isomorphism holds as every complex of flat cotorsion modules over any ring is semi-cotorsion by {\v{S}}{\v{t}}ov{\'{\i}}{\v{c}}ek \cite[Theorem 5.4]{Sto14} or Bazzoni, Cort\'{e}s-Izurdiaga, and Estrada \cite[Theorem 1.3]{BIE19} (or Remark \ref{semi-cotorsion property} if $\dim R<\infty$), along with \eqref{HomKD2}; see \cite[Proposition 2.5]{NY18} for the second isomorphism.

\begin{proof}[Proof of Lemma \ref{cosuppinclusion}.]
By \eqref{derived fc_lambda and fc_coloc}, we obtain the next isomorphisms:
\begin{align}\label{remove_derive}
\LLambda^{\fp}\RHom_R(R_\fp,Y)\cong \LLambda^{\fp}\Hom_R(R_\fp,Y)\cong \Lambda^{\fp}\Hom_R(R_\fp,Y).
\end{align}
Hence, by \eqref{fc_lambda}, \eqref{fc_coloc}, and \eqref{cosupport characterization}, if $\fp\in \cosupp_RY$ then $\fp\in \bigcup_{i\in\mathbb{Z}}\cosupp_RY^i$. 
\end{proof}

Let $X$ be an $R$-complex. 
If $\H^i(X)=0$ for $i\ll 0$ and $I$ is a minimal semi-injective resolution of $X$, then $\supp_RX=\bigcup_{i\in\mathbb{Z}}\supp_RI^i$, as essentially shown by Foxby \cite{Fox79}. However, Chen and Iyengar provide in \cite{CI10} an example of an unbounded $R$-complex $X$ whose minimal semi-injective resolution $I$ satisfies $\supp_RX\subsetneq \bigcup_{i\in\mathbb{Z}}\supp_RI^i$. 

Similarly, there exists a complex $X$ with minimal semi-flat-cotorsion replacement $Y$ such that $\cosupp_RX\subsetneq \bigcup_{i\in \ZZ}\cosupp_R Y^i$, see Example \ref{counterexample}. In particular, this yields a counterexample to the statement of \cite[Theorem 2.7]{Tho18cosupp}; indeed, the argument in \cite[p. 257, l. 5]{Tho18cosupp} is incorrect. The author of that paper sincerely apologizes for his mistake. To redeem this result, we prove in Theorem \ref{correct} a correction (and improvement) to \cite[Theorem 2.7]{Tho18cosupp}. In particular, our correction is sufficient to verify all results of \cite{Tho18cosupp} that use \cite[Theorem 2.7]{Tho18cosupp}.

\begin{thm}\label{correct}
Let $Y$ be a minimal complex of flat cotorsion $R$-modules. 
Suppose that one of the following conditions holds:
\begin{enumerate}
\item[\textnormal{(1)}] $\Hom_R(R_\fp, Y)$ is semi-flat for all $\fp \in \Spec R$;
\item[\textnormal{(2)}] $Y^i=0$ for $i\ll0$.
\end{enumerate}
Then one has an equality:
$$\cosupp_R Y = \bigcup_{i\in \mathbb{Z}} \cosupp_R Y^i.$$
\end{thm}
\noindent
In particular, this shows that if $X$ is an $R$-complex having a minimal semi-flat-cotorsion replacement $Y$ that satisfies either condition (1) or (2) in Theorem \ref{correct}, then $\cosupp_R X$ agrees with the set of prime ideals appearing in $Y$.

\begin{rmk}
The assumption of (1) is an analogy of \cite[Proposition 2.1]{CI10}. Clearly it is satisfied if $Y^i=0$ for $i\gg0$. Moreover, it is also satisfied when $R$ is regular, see \cite[Theorem 1.2 and Proposition 3.3]{II09}.

The condition (2) is the same as the original statement. To salvage this case, we need the next lemma; it essentially follows from a result of Auslander and Buchsbaum \cite{AB58} on finitistic flat dimension; this was reproved in Corollary \ref{FFD} above.
\end{rmk}

\begin{lem}\label{minimal left}
Let $Y$ be a minimal complex of flat cotorsion $R$-modules. Assume that $Y$ is acyclic and $Y^i=0$ for $i\ll0$. Then $Y=0$ in $\C(R)$.
\end{lem}

\begin{proof}
Suppose that $Y\neq 0$ in $\C(R)$ and deduce a contradiction. We can take a prime ideal $\fp\in \bigcup_{i\in\mathbb{Z}}\cosupp_RY^i$. Since $Y$ is acyclic, so is $\Lambda^\fp\Hom_R(R_\fp,Y)$ by (\ref{remove_derive}); this complex satisfies the same condition as $Y$, and so we may replace $Y$ by $\Lambda^{\fp}\Hom_R(R_\fp,Y)$. Thus we may assume $Y$ is a complex consisting of flat cotorsion modules with cosupport in $\{\fp\}$; in particular, $Y$ is an $R_\fp$-complex. Moreover, without loss of generality, we may assume that $Y^i=0$ for $i<0$ and $Y^0\neq 0$. 

Now, fix an integer $n>\dim R_\fp$ and consider the truncation
\[
Y'=(\xymatrix{\cdots \ar[r] & 0\ar[r] &  Y^0\ar[r] & \cdots\ar[r] &  Y^{n}\ar[r] & 0\ar[r] & \cdots}).
\]
Since $Y$ is acyclic, $Y'$ can be regarded as a flat resolution of $C=\coker(d_Y^{n-1})$ over $R_\fp$.  Minimality of $Y$ implies that $\kappa(\fp) \otimes_{R_\fp}Y$ has zero differential by Theorem \ref{criteria}, and hence $\kappa(\fp) \otimes_{R_\fp} Y'$ has zero differential as well.  Thus $\Tor_{n}^{R_\fp}(\kappa(\fp),C)\cong \kappa(\fp) \otimes_{R_\fp} Y^0\neq 0$.  This implies that $\dim R_\fp<n= \fd_{R_\fp} C<\infty$, contradicting that the finitistic flat dimension of $R_\fp$ is at most $n$; see for example Corollary \ref{FFD}.
\end{proof}

\begin{proof}[Proof of Theorem \ref{correct}]
Let $\fp\in \Spec R$, and assume that $\fp\notin \cosupp_RY$. By Lemma \ref{cosuppinclusion}, we only have to show that $\fp\notin \bigcup_{i\in \mathbb{Z}} \cosupp_R Y^i$.

Suppose that condition (1) holds. We then have by \eqref{derived fc_lambda and fc_coloc} that
$$\kappa(\fp)\Lotimes_R\RHom_R(R_\fp, Y)\cong \kappa(\fp)\Lotimes_R\Hom_R(R_\fp, Y)\cong \kappa(\fp)\otimes_R\Hom_R(R_\fp, Y).$$
Minimality of $Y$ implies that $\kappa(\fp)\otimes_R\Hom_R(R_\fp, Y)$ has zero differential by Theorem \ref{criteria}. 
In addition, $\kappa(\fp)\otimes_R\Hom_R(R_\fp, Y)$ is acyclic by (\ref{cosupport characterization}) since $\fp\notin \cosupp_RY$.
It follows that $\kappa(\fp)\otimes_R\Hom_R(R_\fp, Y)=0$ in $\C(R)$. Hence $\fp\notin \bigcup_{i\in \mathbb{Z}} \cosupp_R Y^i$ by (\ref{fc_kappa}).

Next suppose that condition (2) holds.  Together, (\ref{cosupport characterization}) and (\ref{remove_derive}) yield that the complex $\Lambda^{\fp}\Hom_R(R_\fp, Y)$ is an acyclic complex of flat cotorsion modules. Further, as $Y^i=0$ for $i\ll0$, we also have $(\Lambda^{\fp}\Hom_R(R_\fp, Y))^i=0$ for $i\ll0$. Minimality of $Y$ implies $\Lambda^\fp\Hom_R(R_\fp,Y)$ is minimal, by Theorem \ref{criteria}, hence Lemma \ref{minimal left} yields $\Lambda^{\fp}\Hom_R(R_\fp, Y)=0$ in $\C(R)$, that is, $\fp\notin \bigcup_{i\in \mathbb{Z}} \cosupp_R Y^i$. 
\end{proof}

\begin{rmk}
Using finitistic injective dimension (see \cite[Theorem 2.4]{AB58} and \cite[Theorem 1]{Mat59}) it can also be shown that if $I$ is a minimal $R$-complex of injective modules with $I^i=0$ for $i\gg0$, then there is an equality $\supp_R I = \bigcup_{i\in \mathbb{Z}} \supp_R I^i$. Compare this with \cite[Proposition 2.1]{CI10}.
\end{rmk}

Our next task is to give an example of an $R$-complex $X$ whose minimal semi-flat-cotorsion replacement $Y$ satisfies $\cosupp_R X \subsetneq \bigcup_{i\in\mathbb{Z}}\cosupp_RY^i$. Although our example is analogous to \cite[Proposition 2.7]{CI10}, a key role is played by Construction \ref{Dim1orCountable} and the following result which gives a condition for this construction to yield a minimal complex.

\begin{lem}\label{const_min}
Assume $\dim R\leq 1$ or $R$ is countable. Let $P$ be a complex of projective $R$-modules such that $R/\fp\otimes_R P$ has zero differential for every minimal prime $\fp$. The complex $Y_P$ in Construction \ref{Dim1orCountable} is minimal.
\end{lem}

\begin{proof}
It is enough to show that 
$\kappa(\fq)\otimes_R\Hom_R(R_\fq, Y_P)$ has zero differential for every prime ideal $\fq$ of $R$, by Theorem \ref{criteria}. Denote by $W$ the set of maximal ideals of $R$.  For $\fn\in W$, application of $\Lambda^{\fn}$ to the exact sequence
\[\xymatrix{
0\ar[r] &P\ar[r]& \prod_{\fm\in W} \Lambda^{\fm}P \ar[r] & (\prod_{\fm\in W} \Lambda^{\fm}P)/P\ar[r] & 0
}\]
preserves exactness per \cite[\S 4, p. 69]{Lip02} and sends the map $P\to \prod_{\fm\in W}\Lambda^\fm P$ to an isomorphism, see \eqref{fc_lambda}. It follows that the complex $C=(\prod_{\fm\in W} \Lambda^{\fm}P)/P$ consists of flat cotorsion modules with cosupport in $(\Spec R)\setminus W$, so by \eqref{fc_kappa},
\begin{align}
\kappa(\fq)\otimes_R\Hom_R(R_\fq,Y_P)\cong \begin{cases} R/\fq\otimes_R\Hom_R(R_\fq,\Lambda^{\fq}P)\text{, if $\fq\in W$,} \\
R/\fq\otimes_R\Hom_R(R_\fq,C[-1])\text{, if $\fq\not\in W$}.\end{cases}\label{cases}
\end{align}

Now fix $\fq\in \Spec R$. Application of $R/\fq\otimes_R-$ to the canonical surjection ${\prod_{\fm\in W}\Lambda^{\fm}P \to C}$ yields a surjective chain map
\begin{align}\label{surjective chain map}
\xymatrix{R/\fq\otimes_R\textstyle{\prod_{\fm\in W}}\Lambda^{\fm}P \ar[r] & R/\fq\otimes_RC}.
\end{align}
Moreover, it holds that
$$R/\fq\otimes_R\textstyle{\prod_{\fm\in W}} \Lambda^{\fm}P\cong \textstyle{\prod_{\fm\in W}} (R/\fq\otimes_R\Lambda^{\fm}P)\cong \textstyle{\prod_{\fm\in W}} \Lambda^{\fm}(R/\fq\otimes_RP),$$
see \cite[Lemma 2.3]{NY18} for the second isomorphism.
Taking a minimal prime ideal $\fp$ with $\fp \subseteq \fq$, we have $R/\fq\otimes_RP\cong R/\fq \otimes_R R/\fp\otimes_RP$, therefore the assumption on $P$ implies that $\Lambda^{\fm}(R/\fq\otimes_RP)$ has zero differential, and so does $\prod_{\fm\in W} \Lambda^{\fm}(R/\fq\otimes_RP)$.
Thus both complexes appearing in \eqref{surjective chain map} have zero differential.

To complete the proof, apply $\Hom_R(R_\fq,- )$ to \eqref{surjective chain map}:
$$\Hom_R(R_\fq, R/\fq\otimes_R\textstyle{\prod_{\fm\in W}}\Lambda^{\fm}P) \to \Hom_R(R_\fq, R/\fq\otimes_RC),
$$
where the both complexes have zero differential.
Regarding this chain map as a double complex, we see from \eqref{fc_tensor} that its total complex  is nothing but the complex $\kappa(\fq)\otimes_R\Hom_R(R_\fq,Y_P)$.
Going back to \eqref{cases}, we conclude that $Y_P$ is minimal.
\end{proof}

We now give a counterexample to \cite[Theorem 2.7]{Tho18cosupp}.

\begin{exa}\label{counterexample}
Let $k$ be a field and $R=k[[x,y]]/(x^2)$.  Set $\fm=(x,y)$ and $\fp=(x)$. We construct an $R$-complex $X$ and a semi-flat-cotorsion replacement $Y_P$ of $X$ such that $\cosupp_R X=\{\fm\}$ and $\bigcup_{i\in \ZZ} \cosupp_R Y_P^i = \{\fp,\fm\}$. Let $M=R/(x)$. As $R$ is $\fm$-adically complete, $R$ is flat cotorsion hence the complex
$$F=(\xymatrix{\cdots \ar[r]^{x} & R\ar[r]^{x} & R\ar[r]^{x} & R\ar[r] & 0})$$
is a minimal resolution of $M$ by flat cotorsion $R$-modules.
Theorem \ref{correct} implies $\cosupp_RM=\{\fm\}$. 
We set $P=\bigoplus_{i\in \ZZ}F[i]$ and $X=\bigoplus_{i\in \ZZ}M[i]$. The quasi-isomorphism $F\xrightarrow{} M$ induces a quasi-isomorphism $P\xrightarrow{} X$. Furthermore, since each $F[i]$ is semi-flat, one obtains that $P$ is also semi-flat.

The differential of $P$ is given by multiplication by $x$, hence $R/\fp\otimes_R P$ has zero differential. Construction \ref{Dim1orCountable} applied to the complex $P$ yields a quasi-isomorphism $P\to Y_P$ with pure acyclic mapping cone, where $Y_P$ is semi-flat-cotorsion by construction and minimal by Lemma \ref{const_min}. 
Furthermore, for each $n\in \ZZ$, one has $P^n\cong \bigoplus_{\NN} R$. Since $\dim R>0$, the direct sum $\bigoplus_{\NN} R$ is not isomorphic to its $\fm$-adic completion $\Lambda^{\fm}(\bigoplus_{\NN} R)$. Thus the quotient module $\Lambda^{\fm}(\bigoplus_{\NN} R)/(\bigoplus_{\NN} R)$ is non-trivial, and we see from the proof of Lemma \ref{const_min} that this is a flat cotorsion module with cosupport in $\{\fp\}$. 
Therefore, the minimal semi-flat-cotorsion replacement $Y_P$ of $X$ contains non-trivial flat cotorsion modules with cosupport in both $\{\fm\}$ and $\{\fp\}$. In other words,  $\bigcup_{i\in \ZZ} \cosupp_RY^i_P=\{\fp, \fm\}$.

However, we claim $\cosupp_RX=\{\fm\}$. Indeed, there is an isomorphism of complexes $X\cong \prod_{i\in \ZZ}M[i]$, and hence $\cosupp_RX=\bigcup_{i\in \ZZ} \cosupp_R M[i]=\{\fm\}.$
Consequently we have 
$$\cosupp_RX=\cosupp_R Y_P\subsetneq \bigcup_{i\in \ZZ} \cosupp_RY^i_P.$$

We further point out that $\Hom_R(R_\fp,-)$ may not preserve semi-flatness for complexes of flat cotorsion modules. Indeed, although $Y_P$ is semi-flat, if $\Hom_R(R_\fp,Y_P)$ was also semi-flat then the above strict containment would contradict the conclusion of Theorem \ref{correct}.
\end{exa}

\appendix
\section{Semi-flat-cotorsion replacements for associative rings}\label{appendix_sfc}
\noindent
In this appendix, let $A$ be an associative ring with identity. Here, left $A$-modules and complexes of left $A$-modules are simply referred to as $A$-modules and $A$-complexes, respectively.
Let $\C(A)$ denote the category of $A$-complexes, $\K(A)$ the homotopy category of $A$-complexes, and $\D(A)$ the derived category over $A$.
We first recall what it means for an $A$-complex $X$ to be semi-projective, semi-injective, or semi-flat;\footnote{In the literature, the prefix ``DG-'' is also used in place of ``semi-''. We follow notation of \cite{AFH}.} these have assumptions on the components of $X$ in addition to being $K$-projective, $K$-injective, or $K$-flat in the sense of Spaltenstein \cite{Spa88}.
\begin{itemize}
\item $X$ is {\em semi-projective} if $\Hom_A(X,-)$ preserves acyclicity and $X^i$ is projective for every $i\in \ZZ$. 
\item $X$ is {\em semi-injective} if $\Hom_A(-,X)$ preserves acyclicity and $X^i$ is injective for every $i\in \ZZ$.
\item $X$ is {\em semi-flat} if $-\otimes_AX$ preserves acyclicity and $X^i$ is flat for every $i\in \ZZ$.  
\end{itemize}

An $A$-module $C$ is called \emph{cotorsion} if $\Ext_A^1(F,C)=0$ for all flat $A$-modules $F$. There is a natural corresponding notion---see Enochs and Garc\'{i}a Rozas \cite[Definition 3.3 and Proposition 3.4]{EGR98}---for a complex of cotorsion $A$-modules as well: an $A$-complex $X$ is {\em semi-cotorsion} if $\Hom_A(-,X)$ preserves acyclicity of pure acyclic complexes of flat $A$-modules and $X^i$ is cotorsion for every $i\in \ZZ$; recall that a complex of flat $A$-modules is pure acyclic if and only if it is acyclic and semi-flat. A standard argument shows that an $A$-complex $C$ such that $C^i$ is cotorsion for all $i\in \ZZ$ and $C^i=0$ for $i\ll0$ is semi-cotorsion.

For a semi-flat $A$-complex $F$ and semi-cotorsion $A$-complex $C$, we have an isomorphism in $\D(A)$:
\begin{align}\label{HomKD2}
\RHom_A(F,C)\cong\Hom_A(F,C);
\end{align}
this follows by noting that the mapping cone of a semi-projective resolution $P\to F$ is pure acyclic. It then follows from \eqref{HomKD2} that
\begin{align}\label{HomKD}
\Hom_{\D(A)}(F,C)\cong\Hom_{\K(A)}(F,C).
\end{align}
In particular, a morphism in $\D(A)$ between $A$-complexes that are both semi-flat and semi-cotorsion can be realized by a morphism in $\K(A)$. This naturally leads us to make the next definition:

\begin{dfn}\label{dfn_sfc}
An $A$-complex $X$ is \emph{semi-flat-cotorsion} if $-\otimes_AX$ preserves acyclicity, $\Hom_A(-,X)$ preserves acyclicity of pure acyclic complexes of flat modules, and $X^i$ is flat cotorsion for every $i\in \ZZ$.
\end{dfn}
\noindent 
In other words, an $A$-complex is semi-flat-cotorsion if and only if it is semi-flat and semi-cotorsion.  

Recall that an $A$-complex $X$ is said to be \emph{minimal} if every homotopy equivalence $X\to X$ is an isomorphism in $\C(A)$; see \cite{AM02}. The next lemma follows from \eqref{HomKD} and the definition of minimality.
\begin{lem}\label{sfc_unique}
Let $X$ and $Y$ be minimal semi-flat-cotorsion complexes that are isomorphic in $\D(A)$. Then $X\cong Y$ in $\C(A)$.
\end{lem}
\noindent
In particular, the zero complex is the only acyclic minimal semi-flat-cotorsion complex.

Although every complex has a semi-flat resolution, not every complex has a semi-flat-cotorsion resolution (see Example \ref{no_res}); instead, we consider the following natural notion:
\begin{dfn}\label{sfcreplacement}
A \emph{semi-flat-cotorsion replacement} of an $A$-complex $X$ is an isomorphism in $\D(A)$ between $X$ and a semi-flat-cotorsion $A$-complex.
\end{dfn}

The next result is due to Gillespie \cite{Gil04}.

\begin{thm}\label{Gillespie_exist}
Every $A$-complex $X$ has a semi-flat-cotorsion replacement. 
\end{thm}
\begin{proof}
Let $F\xrightarrow{} X$ be a semi-flat resolution. By \cite[Corollary 4.10]{Gil04}\footnote{Although the result is stated for commutative rings, it is well-known that Gillespie's argument holds without this assumption; see also \cite{YL11}.}, the pair of pure acyclic complexes of flat modules and semi-cotorsion complexes forms a complete cotorsion pair on the category of $A$-complexes; in particular, this implies there is an exact sequence of $A$-complexes,
$$0\to F\to Y\to P\to 0,$$
where $Y$ is semi-cotorsion and $P$ is a pure acyclic complex of flat modules. The complex $P$ is semi-flat, hence $Y$ is semi-flat as well. It now follows that $Y$ is a semi-flat-cotorsion replacement of $X$.
\end{proof}

Analogous to the roles of semi-projective complexes and semi-injective complexes, \eqref{HomKD} and Theorem \ref{Gillespie_exist} show that semi-flat-cotorsion complexes also describe the derived category:
\begin{cor}
The homotopy category of semi-flat-cotorsion $A$-complexes is equivalent to $\D(A)$.
\end{cor}

{\v{S}}{\v{t}}ov{\'{\i}}{\v{c}}ek \cite[Theorem 5.4]{Sto14} shows that every complex of flat cotorsion $A$-modules is semi-cotorsion; in fact, a recent result of Bazzoni, Cort\'{e}s-Izurdiaga, and Estrada shows that every complex of cotorsion $A$-modules is semi-cotorsion \cite[Theorem 1.3]{BIE19}.  We now have the following characterization: 

\begin{lem}\label{semiflat_cot}
An $A$-complex $X$ is semi-flat-cotorsion if and only if $X$ is semi-flat and $X^i$ is cotorsion for every $i\in \ZZ$.
\end{lem}
\begin{proof}
The forward implication is trivial. The converse is by \cite[Theorem 5.4]{Sto14} or \cite[Theorem 1.3]{BIE19}.
\end{proof}

As a consequence of this lemma, a complex $X$ of flat cotorsion $A$-modules such that $X^i=0$ for $i\gg0$ is semi-flat-cotorsion.

\begin{rmk}
Another important role is played by complexes of flat cotorsion modules; they describe the pure derived category of flat modules, as defined by Murfet and Salarian \cite{MS11}, whose work was motivated by Neeman \cite{Nee08}. The pure derived category is defined as the Verdier quotient of the homotopy category of complexes of flat $A$-modules by the subcategory of pure acyclic complexes. Gillespie's result \cite[Corollary 4.10]{Gil04} implies that the pure derived category may be identified with a subcategory of the homotopy category of complexes of flat cotorsion $A$-modules, see also \cite[Lemma 5.1 and Theorem 6.6]{Gil16}. If any flat $A$-module has finite projective dimension, as over a commutative noetherian ring of finite Krull dimension, then it is not hard to see (without using Lemma \ref{semiflat_cot}) that any pure acyclic complex of flat cotorsion $A$-modules is contractible; this assumption implies that the pure derived category coincides with the homotopy category of complexes of flat cotorsion $A$-modules. Indeed, these two categories are equivalent over any ring by \cite[Corollary 5.8]{Sto14} or \cite[Theorem 1.3]{BIE19}.
\end{rmk}

By Corollary \ref{mindecomposition}, over a commutative noetherian ring of finite Krull dimension, each complex of flat cotorsion modules decomposes as a direct sum of a minimal one and a contractible one. However, we do not know whether this holds in general.

\begin{qst}\label{appendix question0}
Let $Y$ be a complex of flat cotorsion $A$-modules. Does $Y$ decompose as a direct sum $Y'\oplus Y''$, where $Y'$ is minimal and $Y''$ is contractible?  
\end{qst}

For special complexes, this question has an affirmative answer.

\begin{prp}\label{specialcase}
Let $F$ be a complex of flat $A$-modules and assume the canonical surjections $F^{i}\to \im(d_F^{i})$ and $F^{i+1}\to \coker(d_F^{i})$ are flat precovers for all $i\in \ZZ$. Then $F=F'\oplus F''$, where $F'$ is minimal and $F''$ is contractible. 

In particular, if $Y$ is an acyclic complex of flat cotorsion $A$-modules, then we have $Y=Y'\oplus Y''$, where $Y'$ is minimal and $Y''$ is contractible.
\end{prp}
\begin{proof}
For each $i\in \ZZ$, consider the exact sequence 
\[\xymatrix{
0\ar[r] & \im(d_F^{i}) \ar[r] & F^{i+1} \ar[r] & \coker(d_F^{i}) \ar[r] & 0,
}\]
and take a flat cover $P^{i+1}\to \coker(d_F^{i})$. The flat precover $F^{i+1}\to \coker(d_F^{i})$ factors through this flat cover via some split surjection $F^{i+1}\to P^{i+1}$, hence we may write $F^{i+1}=P^{i+1}\oplus Q^{i+1}$.
Observe from the exact sequence that the injection $Q^{i+1}\to F^{i+1}$ lifts to an injection $Q^{i+1}\to \im(d_F^{i})$, which also splits. Using the flat precover $F^{i}\to \im(d_F^{i})$ and the exact sequence, we can give a split injection from the contractible complex $Q(i)=(0\to Q^{i+1}\xrightarrow{=}Q^{i+1}\to 0)$ into $F$, where $Q(i)$ is concentrated in degrees $i$ and $i+1$. 
The split injections $Q(i)\to F$ for each $i\in \mathbb{Z}$ induce a map $\bigoplus_{i\in \ZZ} Q(i)\to F$, which one can further check is a split injection, and whose quotient $F'$ has the property that the canonical surjection $(F')^{i+1}\to \coker(d^{i}_{F'})$ is a flat cover for each $i\in \mathbb{Z}$ by construction. Thus $F=F'\oplus F''$ where $F''=\bigoplus_{i\in \ZZ} Q(i)$ is evidently contractible and $F'$ is minimal by \cite[Theorem 4.1]{Tho19min}\footnote{Although \cite[Theorem 4.1]{Tho19min} implicitly assumes the ring is commutative noetherian, this assumption is not needed in the proof.}.

The second assertion follows from the first along with \cite[Theorem 5.4]{Sto14} or \cite[Theorem 1.3]{BIE19}.
\end{proof}

When $A$ is a left perfect ring, the flat modules are the projective modules, hence every surjection from a flat module is a flat precover.  In this case, Proposition \ref{specialcase} provides an affirmative answer to Question \ref{appendix question0}; its proof is modelled on an argument dual to \cite[Appendix B]{Kra05}.

Complementary to the last statement of Proposition \ref{specialcase}, one may also consider a restriction of Question \ref{appendix question0} to the case of semi-flat-cotorsion complexes. Solving the restricted question is thus equivalent to showing the existence of minimal semi-flat-cotorsion replacements for all complexes, by \eqref{HomKD} and the definition of minimality.

\section*{Acknowledgments}
\noindent
This work was finished during the second author's visit to the University of Verona, with support provided by the Norwegian University of Science and Technology.  The first author was supported by the Program Ricerca di Base 2015 of the University of Verona. The authors would like to thank Sergio Estrada, Srikanth Iyengar, Jan {\v{S}}{\v{t}}ov{\'{\i}}{\v{c}}ek, Ryo Takahashi, and Siamak Yassemi for helpful comments and conversations regarding the topics of this paper.


\end{document}